\documentclass[11pt, reqno]{article}

\usepackage[utf8]{inputenc}
\usepackage[T1]{fontenc}
\usepackage[english]{babel}
\usepackage{lmodern}

\usepackage{geometry}
\usepackage[backend=biber, maxbibnames=99]{biblatex}
\usepackage{csquotes}
\usepackage{url}

\usepackage{mathtools}
\usepackage{amssymb}
\usepackage{amsthm}
\usepackage{wasysym}
\usepackage{latexsym}

\usepackage{setspace}

\usepackage{nicefrac}

\usepackage{relsize}

\usepackage{extarrows}
\usepackage{stmaryrd}

\usepackage[new]{old-arrows}

\usepackage{accents} 

\usepackage{graphicx}
\usepackage{float}
\usepackage{floatflt}
\usepackage{rotating}
\usepackage{color}
\usepackage[absolute]{textpos}

\usepackage{here} 

\usepackage{epic}
\usepackage[percent]{overpic}
\usepackage{caption}
\usepackage{subcaption}
\usepackage{capt-of}

\usepackage{tikz}
\usetikzlibrary{matrix}
\usetikzlibrary{cd}

\tikzset{commutative diagrams/.cd,
	mysymbol/.style={start anchor = center, end anchor = center, draw = none}
}

\newsavebox{\pullbacksquare}
\sbox\pullbacksquare{%
	\begin{tikzpicture}%
		\draw (0,0) -- (1ex,0ex);%
		\draw (1ex,0ex) -- (1ex,1ex);%
\end{tikzpicture}}

\usepackage{adjustbox}

\usepackage{xcolor}
\definecolor{green-new}{rgb}{0.0,0.5,0.0}
\definecolor{green-variant}{rgb}{0.1,0.7,0.2}

\usepackage{multirow}

\usepackage{enumitem}

\usepackage{chngcntr}
\counterwithin{figure}{section}
\counterwithin{table}{section}

\usepackage{scalerel}
\usepackage[usestackEOL]{stackengine} 

\usepackage{xparse}

\usepackage{needspace} 

\usepackage{comment}

\usepackage[linktoc=all, bookmarks, hyperfootnotes=false, psdextra=true]{hyperref}

\hypersetup{
	colorlinks	= true, 
	linkcolor	= green-variant,
	urlcolor	= blue,
	citecolor	= blue,
}

\usepackage{cleveref}

\usepackage{thmtools}

\crefformat{footnote}{#2\footnotemark[#1]#3} 

\newcommand{\N}{\mathbb{N}}

\newcommand{\cC}{\mathcal{C}}

\newcommand{\cG}{\mathcal{G}}

\newcommand{\cP}{\mathcal{P}}

\newcommand{\wt}[1]{\widetilde{#1}}

\newcommand{\co}{\colon\thinspace}
\newcommand{\comp}{\mathbin{\circ}} 

\newcommand{\abs}[1]{\left\lvert #1 \right\rvert}

\newcommand{\id}[1]{\textup{\textsf{id}}_{#1}} % id for identity.
\newcommand{\Set}{\textup{\textsf{Set}}}
\newcommand{\timeT}{\textup{\textsf{T}}}
\newcommand{\timeS}{\textup{\textsf{S}}}
\newcommand{\Pe}{\textup{\textsf{Pe}}}
\newcommand{\Cu}{\textup{\textsf{Cu}}}

\newcommand{\Grph}{\textup{\textsf{Grph}}}
\newcommand{\Grefl}{\textup{\textsf{Grph}}_\textup{\textsf{refl}}}

\newcommand{\intJ}{\int\!J}

\DeclareFontFamily{U}{mathx}{}
\DeclareFontShape{U}{mathx}{m}{n}{<-> mathx10}{}
\DeclareSymbolFont{mathx}{U}{mathx}{m}{n}
\DeclareMathAccent{\widehat}{0}{mathx}{"70}
\DeclareMathAccent{\widecheck}{0}{mathx}{"71}

\newcommand{\Ob}{\textup{Ob}}
\newcommand{\op}{\textup{op}}

\makeatletter
\newcommand*\bigcdot{\mathpalette\bigcdot@{.5}}
\newcommand*\bigcdot@[2]{\mathbin{\vcenter{\hbox{\scalebox{#2}{$\m@th#1\bullet$}}}}}
\makeatother

\newcommand\sBig[1]{\scalerel*[5.5pt]{\Big#1}{
		\ensurestackMath{\addstackgap[1.5pt]{\big#1}}}} 
\newcommand\sBigl[1]{\mathopen{\sBig{#1}}}
\newcommand\sBigr[1]{\mathclose{\sBig{#1}}}

\makeatletter
\newsavebox\myboxA
\newsavebox\myboxB
\newlength\mylenA
\newcommand*\loverline[2][0.75]{%
	\sbox{\myboxA}{$\m@th#2$}%
	\setbox\myboxB\null% Phantom box
	\ht\myboxB=\ht\myboxA%
	\dp\myboxB=\dp\myboxA%
	\wd\myboxB=#1\wd\myboxA% Scale phantom
	\sbox\myboxB{$\m@th\overline{\copy\myboxB}$}%  Overlined phantom
	\setlength\mylenA{\the\wd\myboxA}%   calc width diff
	\addtolength\mylenA{-\the\wd\myboxB}%
	\ifdim\wd\myboxB<\wd\myboxA%
	\rlap{\hskip 0.5\mylenA\usebox\myboxB}{\usebox\myboxA}%
	\else
	\hskip -0.5\mylenA\rlap{\usebox\myboxA}{\hskip 0.5\mylenA\usebox\myboxB}%
	\fi}
\makeatother

\makeatletter
\let\save@mathaccent\mathaccent
\newcommand*\if@single[3]{%
	\setbox0\hbox{${\mathaccent"0362{#1}}^H$}%
	\setbox2\hbox{${\mathaccent"0362{\kern0pt#1}}^H$}%
	\ifdim\ht0=\ht2 #3\else #2\fi
}
\newcommand*\rel@kern[1]{\kern#1\dimexpr\macc@kerna}
\newcommand*\widebar[1]{\@ifnextchar^{{\wide@bar{#1}{0}}}{\wide@bar{#1}{1}}}
\newcommand*\wide@bar[2]{\if@single{#1}{\wide@bar@{#1}{#2}{1}}{\wide@bar@{#1}{#2}{2}}}
\newcommand*\wide@bar@[3]{%
	\begingroup
	\def\mathaccent##1##2{%
		\let\mathaccent\save@mathaccent
		\if#32 \let\macc@nucleus\first@char \fi
		\setbox\z@\hbox{$\macc@style{\macc@nucleus}_{}$}%
		\setbox\tw@\hbox{$\macc@style{\macc@nucleus}{}_{}$}%
		\dimen@\wd\tw@
		\advance\dimen@-\wd\z@
		\divide\dimen@ 3
		\@tempdima\wd\tw@
		\advance\@tempdima-\scriptspace
		\divide\@tempdima 10
		\advance\dimen@-\@tempdima
		\ifdim\dimen@>\z@ \dimen@0pt\fi
		\rel@kern{0.6}\kern-\dimen@
		\if#31
		\overline{\rel@kern{-0.6}\kern\dimen@\macc@nucleus\rel@kern{0.4}\kern\dimen@}%
		\advance\dimen@0.4\dimexpr\macc@kerna
		\let\final@kern#2%
		\ifdim\dimen@<\z@ \let\final@kern1\fi
		\if\final@kern1 \kern-\dimen@\fi
		\else
		\overline{\rel@kern{-0.6}\kern\dimen@#1}%
		\fi
	}%
	\macc@depth\@ne
	\let\math@bgroup\@empty \let\math@egroup\macc@set@skewchar
	\mathsurround\z@ \frozen@everymath{\mathgroup\macc@group\relax}%
	\macc@set@skewchar\relax
	\let\mathaccentV\macc@nested@a
	\if#31
	\macc@nested@a\relax111{#1}%
	\else
	\def\gobble@till@marker##1\endmarker{}%
	\futurelet\first@char\gobble@till@marker#1\endmarker
	\ifcat\noexpand\first@char A\else
	\def\first@char{}%
	\fi
	\macc@nested@a\relax111{\first@char}%
	\fi
	\endgroup
}
\makeatother

\makeatletter
\@ifdefinable\@latex@chi{\let\@latex@chi\chi}
\renewcommand*\chi{{\@latex@chi\smash[t]{\mathstrut}}} % want only bottom half of \mathstrut
\makeatletter

\DeclareMathOperator*{\colim}{colim}

\DeclareMathAlphabet{\mathdutchcal}{U}{dutchcal}{m}{n}
\SetMathAlphabet{\mathdutchcal}{bold}{U}{dutchcal}{b}{n}
\DeclareMathAlphabet{\mathdutchbcal}{U}{dutchcal}{b}{n}

\DeclareMathAlphabet{\newmathbb}{U}{BOONDOX-ds}{m}{n}

\numberwithin{equation}{section}

\theoremstyle{definition}
\newtheorem{defi}[equation]{Definition}

\newtheorem{rmk}[equation]{Remark}
\newtheorem{eg}[equation]{Example}

\theoremstyle{plain}
\newtheorem{thm}[equation]{Theorem}
\newtheorem{prop}[equation]{Proposition}
\newtheorem{lemma}[equation]{Lemma}

\newtheorem{conj}[equation]{Conjecture}

\allowdisplaybreaks[1]

\begin{document}

\title{\normalfont Decomposing time-varying data into simple pieces: structured decompositions of narratives}

\author{
	Benjamin Merlin Bumpus\footnote{Instituto de Matemática e Estatística, Universidade de São Paulo. Rua do Matão, 1010 — 05508–090, São Paulo, SP, Brasil}
	\and Jana K. Nickel\footnote{Department of Mathematics, University of Hamburg, Germany}
}

\date{}

\maketitle

\begin{abstract}
	Graphs that change over time arise throughout applications, but there is no single standard way to decompose them into smaller pieces. In this paper, we propose a systematic categorical method for doing so. The main idea is to combine structured decompositions, which generalize graph decompositions, such as tree-decompositions, with persistent narratives, which model time-varying data as diagrams. We prove that, under suitable categorical hypotheses, any static theory of decompositions can be lifted to a corresponding temporal theory. As case studies, we apply this construction to time-varying graphs and recover natural temporal analogues of ordinary tree-width, complemented tree-width, and the tree-independence number.
\end{abstract}

\subsection*{Funding}
Benjamin Merlin Bumpus was supported by the São Paulo Research Foundation (FAPESP), grant 2025/16921-5.

{
	\hypersetup{linkcolor=black}
	\tableofcontents
}

\section{Introduction}

In graph theory and computer science, it is very common to study a large graph by splitting it into simpler pieces. This is useful for algorithmic applications, such as Courcelle's Theorem~\cite{Courcelle}, and constitutes an essential method in structural proofs. For instance, decompositions play a central role in Robertson and Seymour's celebrated graph structure theorem~\cite{Robertson-Seymour_Graph-minors-XVII}. We refer to such methods broadly as graph decompositions.

\medskip

These ideas are instances of the overarching principle of compositionality: the behaviour of a complex system should be determined by the behaviours of its constituent parts together with the rules by which these parts are connected. The relevance of this principle is reflected in a variety of works in applied category theory, including, for example, \cite{Janssen-Partee, Werning-Hinzen-Machery, Fong, Baez-Pollard, Pollard, Szabo-Thomason, Cicala, Baez-Courser, Baez-Master, Courser, Master, Baez-Courser-Vasilakopoulou, Libkind-et-al, Baez-et-al, Patterson}.

\medskip

To provide a principled way of decomposing arbitrary objects, the authors of~\cite{Bumpus-et-al_Structured-Decompositions} introduced structured decompositions, which give a category-theoretic generalization of tree-decompositions and allow for a systematic treatment of decompositions of data structures. This is part of a broader programmatic question:
\begin{quotation}
	\noindent \textit{Is there a right way of developing theories of decompositions for arbitrary objects of a category?}
\end{quotation}
The aim of the present paper is to give evidence that structured decompositions are candidates for such a method. We do this by applying them to time-varying graphs. More precisely, we show that, starting from an ordinary static theory of decompositions, one can systematically produce a corresponding temporal theory. This procedure recovers natural temporal analogues of familiar graph width parameters and, at the same time, it produces new candidates that arise directly from the general framework.

\medskip

We choose time-varying graphs as our case study for two reasons. First, temporal data is ubiquitous: Most data is not genuinely static, but changes with time. Examples include satellite data, social networks, communication networks, and many forms of biological and epidemiological data. Second, although the theory of time-varying graphs is relatively recent, it has already attracted substantial interest~\cite{Harary-Gupta, Kempe-et-al, Casteigts-et-al, Holme-Saramaeki, Holme, Michail}. This makes it a useful testing ground for a programmatic theory of decompositions. The field is not yet so mature that one can simply fit a general framework to a large collection of established definitions; nevertheless, there are enough natural examples (which we mention below) to test whether the framework recovers the expected first attempts. Our results show that it does.

\medskip

The study of width measures for temporal graphs is still at an early stage, but several recent works already indicate why the problem is subtle. Fluschnik, Molter, Niedermeier, Renken and Zschoche explicitly discuss the difficulty of finding temporal analogues of tree-width and survey several possible approaches and pitfalls~\cite{Fluschnik-et-al}. More recent work has proposed different ways of incorporating time into structural parameters. For instance, Bui-Xuan, Krasnopol, Monasson and Sznajder introduce a derivative construction for temporal graphs and study tree-width and twin-width of this derivative, obtaining model-checking results that avoid using the lifetime itself as a parameter~\cite{Bui-Xuan-et-al}. Enright, Hand, Larios-Jones and Meeks take a different direction, introducing temporal analogues of clique-width, modular-width and neighbourhood diversity for dense temporal graphs~\cite{Enright-et-al_Dense-Temporal-Graphs}, as well as vertex-interval-membership width (introduced by Bumpus and Meeks~\cite{Bumpus-Meeks}) and tree-interval-membership width, which control how much vertex or component activity must be remembered over time~\cite{Enright-et-al_vertex-interval-membership-width}. These parameters have also been used in logical meta-theorems for model checking of first-order and monadic second-order logic
on temporal graphs~\cite{Doering-et-al}. Thus, rather than forming a settled theory, the existing literature provides a small but informative collection of approaches, each emphasizing a different aspect of temporal structure. Our contribution is complementary: We do not claim to identify the definitive temporal analogue of tree-width. Instead, we give a systematic categorical procedure which turns static decomposition theories into temporal ones. In this sense, our framework is closest in spirit to approaches such as derivatives and interval-membership widths, where temporal structure is incorporated into the object being decomposed rather than treated merely as an external parameter.

\medskip

Working directly with temporal graphs as they are typically defined in the literature is difficult. There are many different representations of time-varying data, coming from different fields and motivated by different applications. Moreover, many of these definitions are quite combinatorial, which can make them difficult to generalize. We therefore use the language of persistent narratives: a sheaf-theoretic model of time-varying data. This perspective is especially convenient here because both persistent narratives and structured decompositions are naturally expressed in diagrammatic language. As a result, the two frameworks combine cleanly.

\medskip

We can now state the goals and results of this paper more precisely. We apply structured decompositions to categories of persistent narratives. When we instantiate this construction in graphs, we obtain temporal analogues of several standard graph width notions, including ordinary tree-width, complemented tree-width, and the tree-independence number. Moreover, the construction applies in rather general terms: the main results do not depend on special features of graphs, but on categorical conditions ensuring that structured decompositions and persistent narratives interact well. This is formalized in Theorems~\ref{thm:termporalized-sd-category} and~\ref{thm:temporalization-of-spined-sd-category}.

\medskip

The paper is organized as follows. Section~\ref{section:theoretical-framework} recalls the two ingredients used throughout the paper: persistent narratives and spined structured decomposition categories. In Section~\ref{section:temporalization}, we combine these ingredients and prove the main temporalization theorem. In the following Section~\ref{section:case-studies}, we specialize the general theorem to graph-theoretic case studies, obtaining temporal versions of ordinary tree-width, complemented tree-width, and the tree-independence number. We conclude in Section~\ref{section:conclusion} with directions for future work.

\paragraph{Acknowledgements.}
We would like to thank Daniel Rosiak for helpful initial discussions concerning the sheafification of presheaves of C-sets. The first author would also like to thank the organizers and participants of Dagstuhl Seminar 26251, \emph{Temporal Graphs: Structure, Algorithms, Applications}, for the many insightful discussions that contributed to the development of the ideas in this paper.

\section{Theoretical framework}
\label{section:theoretical-framework}

In this chapter, we will provide background knowledge on basic definitions and concepts with respect to persistent narratives and structured decompositions which we will use throughout.

\subsection{Persistent narratives}
\label{subsection:narratives}

We begin with the definition of persistent narratives, following \cite[Section~2]{Bumpus-et-al_Time-Varying}.
We write $\Set$ for the usual category of sets and maps.

\begin{defi}
	Let $\textup{\textsf{I}}$ denote the subcategory of~$\Set$ whose objects are closed bounded discrete intervals $[a,b]$ in~$\N_0$, that is, subsets $[a,b]\coloneqq\{a,a+1,\dots,b\}\subseteq\N_0$ for integers $0\leq a\leq b$, and whose morphisms are inclusion maps $[a,b]\hookrightarrow[c,d]$ for any integers $0\leq c\leq a\leq b\leq d$. We may regard $\textup{\textsf{I}}$ as a join-semilattice with respect to the partial order of set inclusion where the join of two discrete intervals $[a,b]$ and $[c,d]$ is the discrete interval $[\min\{a,c\},\max\{b,d\}]$.	
\end{defi}

\begin{defi}
	A \textit{discrete time category}~$\timeT$ is a sub-join-semilattice of~$\textup{\textsf{I}}$. We usually refer to its objects as time intervals or merely as intervals (suppressing the adjective ``discrete'').
	We say that $\timeT$ is \textit{finite} iff it contains only finitely many objects.
	An example of a finite discrete time category is illustrated below.
	
	\begin{figure}[H]
		\centering
		
		\begin{equation*}
			\begin{tikzcd}[row sep=small, column sep=small]
				& & & {[0,3]}
				\\
				& & {[0,2]} \arrow[ru, hook] & & {[1,3]} \arrow[lu, hook]
				\\
				& {[0,1]} \arrow[ru, hook] & & {[1,2]} \arrow[lu, hook] \arrow[ru, hook] & & {[2,3]} \arrow[lu, hook]
				\\
				{[0,0]} \arrow[ru, hook] & & {[1,1]} \arrow[lu, hook] \arrow[ru, 	hook] & & {[2,2]} \arrow[lu, hook] \arrow[ru, hook] & & {[3,3]} \arrow[lu, hook]
			\end{tikzcd}
		\end{equation*}
		\caption{A schematic visualization of an example of a finite discrete time category.}
		\label{fig:finite-discrete-time-category}
	\end{figure}
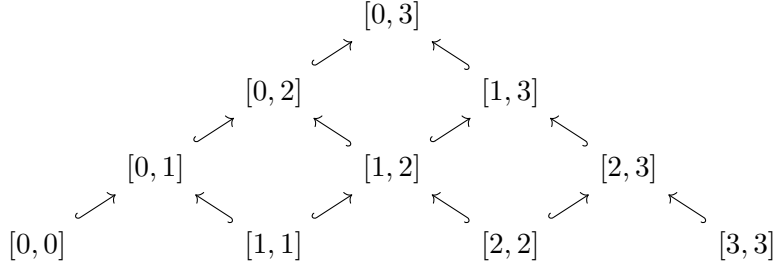
\end{defi}

\begin{defi}\label{defi:narratives}
	Let $\cC$ be a category which admits all pullbacks.
	A \textit{$\cC$-valued persistent narrative with $\timeT$-time}, or a \textit{persistent narrative on~$\timeT$ with values in~$\cC$}, is a presheaf $F\co \timeT^\op\to\cC$ taking any square in~$\timeT$ of the form
	\begin{equation*}
		\begin{tikzcd}
			{[p,p]} \arrow[r, hook] \arrow[d, hook] & {[p,b]} \arrow[d, hook]
			\\
			{[a,p]} \arrow[r, hook] & {[a,b]}
		\end{tikzcd}
	\end{equation*} 
	to a pullback square
	\begin{equation*}
		\begin{tikzcd}
			F([a,b]) \arrow[r] \arrow[d] \arrow[rd, phantom, "\mathlarger{\lrcorner}", very near start] 
			& F([p,b]) \arrow[d]
			\\
			F([a,p]) \arrow[r] 
			& F([p,p])
		\end{tikzcd}
	\end{equation*}
	in the category~$\cC$.
	We denote by $\Pe(\timeT,\cC)$ the full subcategory of the category $[\timeT^\op,\cC]$ of presheaves $\timeT^\op\to\cC$ whose objects are the $\cC$-valued persistent narratives with $\timeT$-time. Hence, the morphisms of $\Pe(\timeT,\cC)$ from $F$ to~$F'$ are the usual natural transformations $\varphi\co F\Rightarrow F'\co \timeT^\op\to\cC$.

\end{defi}

\begin{rmk}
	In~\cite{Bumpus-et-al_Time-Varying}, it is shown that persistent narratives are the same as sheaves $\timeT^\op\to\cC$ where the time category~$\timeT$ carries the Johnstone coverage as introduced in~\cite{Johnstone_coverage}, see also \cite[Section~3.2]{Schultz-et-al_Dynamical}, which makes $\timeT$ into a site.
	Specifically, the Johnstone covering families over an interval $[a,b]\in\Ob(\timeT)$ are pairs of inclusions
	\[ [a,p]\hookrightarrow[a,b] \quad \text{and} \quad [p,b]\hookrightarrow[a,b] \]
	for all integers $p\in[a,b]$.
	Spelling out the corresponding sheaf conditions, one obtains the concrete description of Definition~\ref{defi:narratives} \cite[Proposition~2.7]{Bumpus-et-al_Time-Varying}. For our concerns, this hands-on description turns out to be more convenient than the abstract sheaf-theoretic presentation of persistent narratives.
\end{rmk}

\subsection{Structured decompositions and spined sd-categories}

The definition of a structured decomposition, Definition~\ref{defi:structured-decomposition}, involves the notion of a graph, so to begin with, let us clarify what we understand by a graph.

\begin{defi}
	By a \textit{graph}, we mean a finite simple undirected graph, that is, a pair $J=(V(J),E(J))$ of finite sets where the elements of $E(J)$ are two-elements subsets of~$V(J)$. The elements of $V(J)$ are called the \textit{vertices} of~$J$ and those of $E(J)$ are the \textit{edges} of~$J$. Given an edge $vw\coloneqq\{v,w\}$, we call $v$ and~$w$ its \textit{endvertices}. Note that we do not allow loops (edges consisting of one single vertex) or multiple edges (different edges with the same endvertices). 
	A \textit{graph morphism} is a map of vertex sets preserving adjacency.
	Graphs and morphisms form a category, which we denote by $\Grph$.
\end{defi}

\paragraph{Tree-decompositions}

To acknowledge the benefit of structured decompositions, Definition~\ref{defi:structured-decomposition}, we should first say a few words about ordinary tree-decompositions.
Originally initiated by Halin in 1976~\cite{Halin}, the notion of a tree-decomposition, Definition~\ref{defi:tree-decomposition}, has been rediscovered in 1984 by Robertson and Seymour~\cite{Robertson-Seymour_Graph-minors-III} and since then plays an essential role both within structural and extremal graph theory.
Robertson and Seymour themselves used the concept of a tree-decomposition as a crucial component within their famous graph structure theorem~\cite{Robertson-Seymour_Graph-minors-XVII}, which provides a profound relation between their graph minor theory and the theory of topological embeddings.

Intuitively, a tree-decomposition of a graph reveals its global structure in the case where this is tree-like:

\begin{defi}\label{defi:tree-decomposition}
	A \textit{tree-decomposition} of a graph~$G$ is a pair $(T, (V_t)_{t\in\Ob(T)})$ consisting of a tree~$T$ and a family of vertex sets $V_t\subseteq V(G)$ indexed by the vertices $t$ of~$T$ and satisfying the following two conditions:
	\begin{enumerate}[label=(T\theenumi), leftmargin=35pt]
		\item $G=\bigcup_{t\in\Ob(T)}G[V_t]$ where $G[V_t]$ denotes the subgraph of~$G$ induced by~$V_t$;
		\item For each vertex $v$ of~$G$, the induced subgraph $T_v := T[\{t\in T\mathbin| v\in V_t\}]\subseteq T$ is connected.
	\end{enumerate}
	We call $T$ the \textit{decomposition tree} and the subsets $V_t\subseteq V(G)$ the \textit{bags} of the decomposition.
\end{defi}

An example of a tree-decomposition is illustrated in Figure~\ref{fig:tree-decomposition}.
	\begin{figure}[H]
		\centering
		\includegraphics[width=0.8\textwidth]{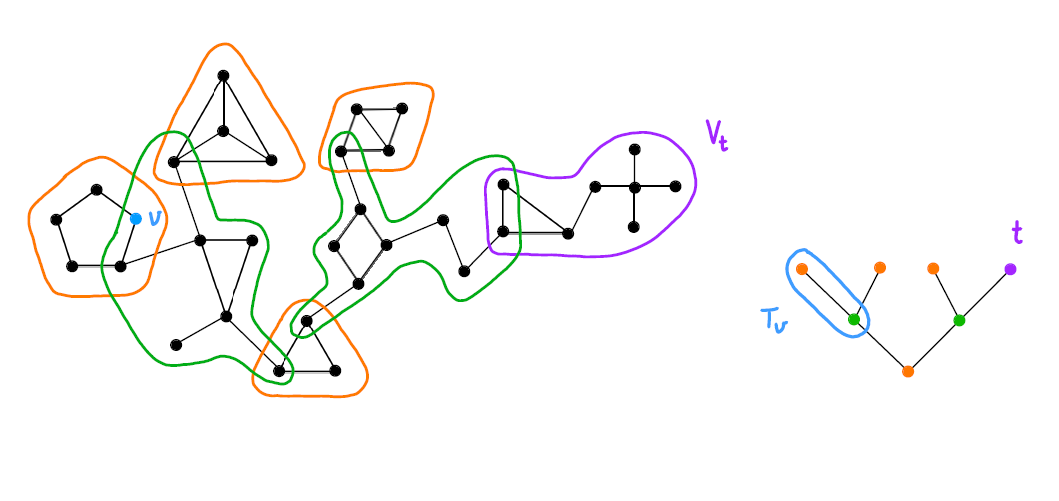}
		\setlength{\abovecaptionskip}{-1cm}
		\caption{A tree-decomposition of a graph (on the left-hand side) and its decomposition tree (on the right-hand side).}
		\label{fig:tree-decomposition}
	\end{figure}

\paragraph{Tree-width.}
\label{paragraph:tree-width}

Tree-decompositions provide a handy tool to constructively compute the \textit{tree-width} of a graph~$G$, which measures the degree to which extent the structure of~$G$ resembles the structure of a tree.

\begin{defi}\label{defi:width-tree-width}
	Let $G$ be a graph.
	The \textit{width} of a tree-decomposition $(T, (V_t)_{t\in\Ob(T)})$ is the maximum order of its bags minus one,
	$\max_{t\in V(T)} \abs{V_t}-1$.
	The \textit{tree-width} $\textup{tw}(G)$ of~$G$ is defined as the minimum width of its tree-decompositions,
	\[ \textup{tw}(G) \coloneqq \min_{(T,(V_t)_{t\in\Ob(T)})} \, \max_{t\in V(T)} \abs{V_t}-1 \]
	where the minimum is taken over all tree-decompositions of~$G$.
\end{defi}
The reason for the ``minus one'' in the definition of the width of a tree-decomposition is to guarantee that the tree-width of any tree itself is precisely one, as we would expect from a reasonable measure of structural tree-resemblance.

\subsubsection{Structured decompositions}
\label{subsubsection:Structured-decompositions}

From a category theoretical point of view, tree-decompositions are instances of structured decompositions, which we will now define, basing on \cite[Section~2]{Bumpus-et-al_Structured-Decompositions}.

\begin{defi}
	Let $J$ be a graph.
	Its \textit{barycentric subdivision} is the category~$\intJ$ defined as follows. The set of objects is the disjoint union $V(J)\sqcup E(J)$ of the vertex set and the edge set of~$J$. Apart from the identity morphisms, $\intJ$~has precisely two morphisms $e_v\co e\to v$ and $e_w\co e\to w$ for each edge $e=vw\in E(J)$. 
	In other words, the category~$\intJ$ is obtained from the graph~$J$ by replacing every vertex of~$J$ by an object and every edge~$e$ of~$J$ with endvertices $v$ and~$w$ by an additional object~$e$ together with a span 
	$v\xleftarrow{e_v}e\xrightarrow{e_w}w$.	
\end{defi}

\begin{eg}
	The barycentric subdivision of the triangle $K^3$ depicted on the left is the category visualized on the right.
	\begin{equation*}
		\begin{tikzcd}[row sep=large]
			& v \arrow[ld, "a", no head, swap] \arrow[rd, "b", no head]
			\\ u \arrow[rr, "c", no head, swap] & & w
		\end{tikzcd}
		\qquad \qquad
		\begin{tikzcd}[row sep=large, column sep=huge]
			a \arrow[d, "a_u", swap] \arrow[rd, "a_v", pos=0.8] & c \arrow[ld, "c_u", pos=0.2, swap] \arrow[rd, "c_w", pos=0.2] & b \arrow[ld, "b_v", pos=0.8, swap] \arrow[d, "b_w"] 
			\\ u & v & w
		\end{tikzcd}
	\end{equation*}
\end{eg}

\begin{defi}\label{defi:structured-decomposition}
	Let $J$ be a graph, and let $\cC$ be a category.
	A \textit{$J$-structured decomposition in~$\cC$} or \textit{with values in~$\cC$} is a functor of the form $d\co \intJ\to\cC$ with the property that for every morphism $k\co x\to y$ in~$\intJ$, its image $d(k)\co d(x)\to d(y)$ under~$d$ is a monomorphism in~$\cC$. The \textit{bags} of~$d$ are the objects~$d(v)$ for $v\in V(J)$.
\end{defi}

\begin{eg}
	Tree-decompositions can be described as tree-shaped structured decompositions with values in the category $\Grph$ of graphs, that is, functors of the form $\int T\to \Grph$ for some tree~$T$. For full details, we refer to \cite[Proposition 3.1.1]{Bumpus-et-al_Structured-Decompositions}, which we will return to in Case study~\ref{subsection:case-study-ordinary-tree-width}. Replacing $T$ by an arbitrary graph leads to the more general notion of a graph-decomposition, studied for example in \cite{Carmesin-et-al_Graph-decompositions} and~\cite{Diestel-et-al_Graph-decompositions}.
	The concept of graph decompositions moreover gives rise to a great variety of combinatorial width parameters measuring the structural resemblance of a graph with respect to a specific graph model, such as a tree, a path or a cycle. Special examples include the classical tree-width as well as diverse variants thereof, as discussed in \cite[Chapter~3]{Bumpus-et-al_Structured-Decompositions}, for instance.
	
	Figure~\ref{fig:example-graphs} shows a structured decomposition with values in~$\Grph$ shaped by the cycle~$C^5$ of length five, together with the cycle itself and its barycentric subdivision~$\int C^5$. The functor $d$ sends those objects $x_i$ of~$\int C^5$ that correspond to the vertices $x_i$ of the graph~$C^5$ to complete graphs.
	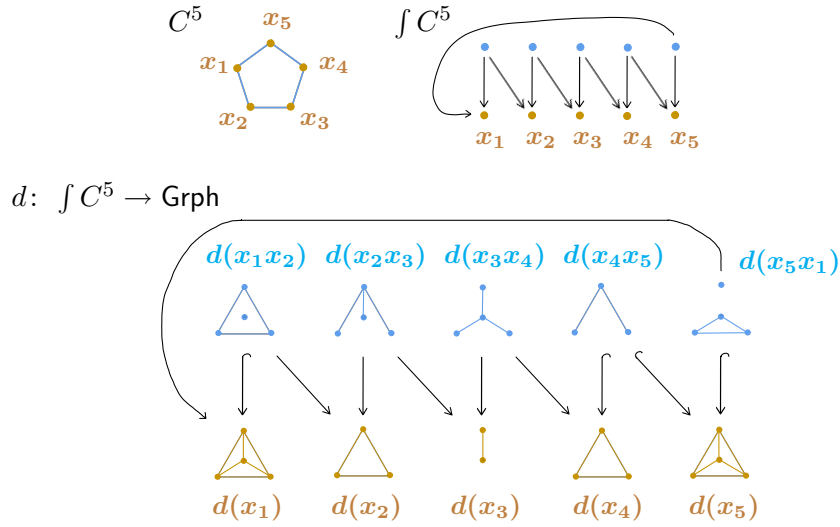
\begin{figure}
		\centering
		\begin{overpic}[tics=10, width=0.5\textwidth, trim={0 11.5cm 0 0}]
			{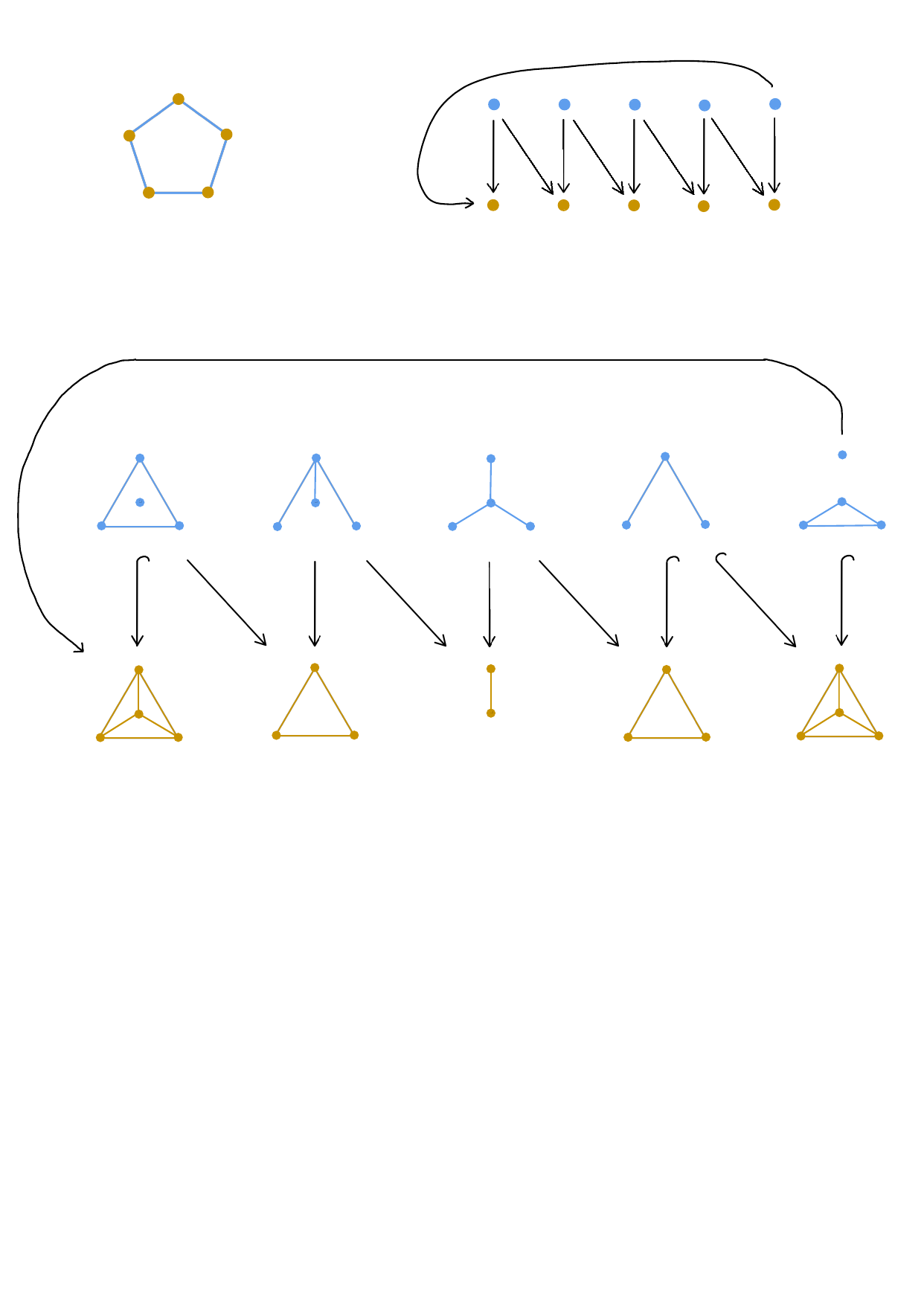}
			\put(3,78){$C^5$}
			\put(8,72){\color{brown}\boldmath{\small$x_1$}}
			\put(11,63){\color{brown}\boldmath{\small$x_2$}}
			\put(24,63){\color{brown}\boldmath{\small$x_3$}}
			\put(27,72){\color{brown}\boldmath{\small$x_4$}}
			\put(18,79){\color{brown}\boldmath{\small$x_5$}}
			\put(39,78){$\int C^5$}
			\put(52,60){\color{brown}\boldmath{\small$x_1$}}
			\put(60,60){\color{brown}\boldmath{\small$x_2$}}
			\put(67.5,60){\color{brown}\boldmath{\small$x_3$}}
			\put(75.5,60){\color{brown}\boldmath{\small$x_4$}}
			\put(83,60){\color{brown}\boldmath{\small$x_5$}}
			\put(-22,50){$d\co \int C^5\to \Grph$}
			\put(9,41){\color{cyan}\boldmath{\small$d(x_1x_2)$}}
			\put(28,41){\color{cyan}\boldmath{\small$d(x_2x_3)$}}
			\put(47,41){\color{cyan}\boldmath{\small$d(x_3x_4)$}}
			\put(66,41){\color{cyan}\boldmath{\small$d(x_4x_5)$}}
			\put(94,40){\color{cyan}\boldmath{\small$d(x_5x_1)$}}
			\put(10,1){\color{brown}\boldmath{\small$d(x_1)$}}
			\put(29,1){\color{brown}\boldmath{\small$d(x_2)$}}
			\put(48,1){\color{brown}\boldmath{\small$d(x_3)$}}
			\put(67.5,1){\color{brown}\boldmath{\small$d(x_4)$}}
			\put(86,1){\color{brown}\boldmath{\small$d(x_5)$}}
		\end{overpic}
		\caption{The cycle~$C^5$, the category~$\int C^5$ and a $C^5$-shaped structured decomposition of graphs.}
		\label{fig:example-graphs}
	\end{figure}
\end{eg}

\subsubsection{Spined sd-categories}

We now come to the definition of a spined structured decomposition category, which plays a central role in these notes.
Spined structured decomposition categories, or briefly, spined sd-categories, were introduced in~\cite{Bumpus-et-al_Structured-Decompositions} and constitute a unified axiomatic setting to study tree-width and generalized width notions.
This section is again based on \cite[Section~2]{Bumpus-et-al_Structured-Decompositions}.

\begin{defi}
	A \textit{structured decomposition category}, or simply, an \textit{sd-category}, is a pair $(\cC,\cG)$ where $\cC$ is a category and $\cG$~is a class of graphs including the trivial graph with exactly one vertex and no edge, subject to the condition
	\begin{enumerate}[leftmargin=0.5in]
	\renewcommand{\labelenumi}{(S\theenumi)}
		\item\label{axiom:S1} Every structured decomposition $d\co\intJ\to\cC$ in~$\cC$ with $J\in\cG$ admits a colimit.
	\end{enumerate} 
	We call $\cG$ the \textit{index class} of the pair $(\cC,\cG)$.
\end{defi}

\begin{defi}\label{defi:spined-sd-category}
	A \textit{spine} on a category~$\cC$ is a family $\Omega = (\Omega_n)_{n\geq0}$ of subcategories $\Omega_n\subseteq\cC$ satisfying the following axioms:
	\begin{enumerate}[leftmargin=0.5in]
	\renewcommand{\labelenumi}{(S\theenumi)}
	\setcounter{enumi}{1}
		\item\label{axiom:S2} For each integer $n\geq0$, $\Omega_n$~is contained in~$\Omega_{n+1}$. In other words, we have an increasing sequence of subcategories
		\[ \Omega_0\subseteq\Omega_1\subseteq\Omega_2\subseteq\dots \]
	
		\item\label{axiom:S3} The union $\bigcup\Omega\coloneqq \bigcup_{n\geq0}\Omega_n$, thus a well-defined subcategory of~$\cC$, is closed under isomorphic objects in~$\cC$, that is, for any objects $W\in\Ob(\bigcup\Omega)$ and $X\in\Ob(\cC)$, if $W$ and~$X$ are isomorphic in~$\cC$, then $X\in\Ob(\bigcup\Omega)$. Furthermore, for any $n\geq0$, $\Omega_n$~is closed under isomorphic objects as a subcategory of~$\bigcup\Omega$.
			
		\item\label{axiom:S4} For every object $X\in\Ob(\cC)$, there exists an object $W\in\Ob(\Omega_n)$ for some integer $n\geq0$ together with a monomorphism $W\rightarrowtail X$ in~$\cC$.
		
		\item\label{axiom:S5} Given $n\geq1$ and objects $W$ of~$\Omega_n$ and $\wt{W}$ of~$\Omega_{n-1}$ along with a monomorphism $W\rightarrowtail \wt{W}$ in~$\cC$, then $W$ is already contained in~$\Omega_{n-1}$.
	\end{enumerate}
	A \textit{spined sd-category} is a triple $(\cC,\cG,\Omega)$ where $(\cC,\cG)$ is an sd-category and $\Omega$ is a spine on~$\cC$.		
\end{defi}

\paragraph{The notion of width in a spined sd-category.}
\label{paragraph:width-in-spined-sd-category}

In an analogous way as tree-decompositions enable a concrete computation of tree-widths, structured decompositions and spined sd-categories provide a sound setting to study generalized width notions. 
In fact, the paper~\cite{Bumpus-et-al_Structured-Decompositions} discusses several concrete examples of different width notions of graph theory, such as ordinary tree-width, complemented tree-width, the tree independence number, hypergraph tree-width and layered tree-width, for instance. The authors prove that they can be recovered by their categorified concept, which shows that spined sd-categories capture the classical notion of tree-width and many variants thereof. This justifies them as a sensible framework to investigate width notions of graph structures in a general, uniform and systematic manner. 
The idea is to imitate the definition of ordinary tree-width from graph theory by replacing tree-decompositions by arbitrary structured decompositions. 
We will later reveal that our temporalized version of spined sd-categories recovers direct temporal generalizations of some example ordinary width notions, see Section~\ref{section:case-studies}.

The rest of the section at hand is based on \cite[Section~2.5]{Bumpus-et-al_Structured-Decompositions}.

\begin{defi}
	For any object~$X$ of a spined sd-category $\Gamma=(\cC,\cG,\Omega)$, we define its \textit{$\Gamma$-size} to be the minimum non-negative integer~$n$ for which there exists an object $W\in\Ob(\Omega_n)$ together with a monomorphism $X\rightarrowtail W$ in~$\cC$. The axiom~(S\ref{axiom:S4}) guarantees that this number is finite, so we obtain a well-defined map
	\[ s_\Gamma\co \Ob(\cC)\to\N_0 \]
	where for any $X\in\Ob(\cC)$, $s_\Gamma(X)$~is the $\Gamma$-size of~$X$.	We call $s_\Gamma$ the \textit{size function} of~$\Gamma$.
\end{defi}

\begin{rmk}
	In \cite[Definition 2.5.4]{Bumpus-et-al_Structured-Decompositions}, the $\Gamma$-size of~$X$ is defined as the minimum integer $n\geq0$ such that $\Omega_n$ is the smallest full subcategory of~$\bigcup\Omega$ with a monomorphism $W\rightarrowtail X$ in~$\cC$ for some $W\in\Ob(\Omega_n)$. We have slightly modified this definition because we will later require that the subcategories $\Omega_n\subseteq\cC$ are full, in which case both definitions agree.
\end{rmk}

This size function now allows to imitate the definition of tree-width of tree-decompositions and of ordinary graphs to a width notion of structured decompositions and of objects in arbitrary spined sd-categories, respectively.

\begin{defi}\label{defi:width-Gamma-width}
	The \textit{width}~$w_\Gamma(d)$ of a structured decomposition $d\co \intJ\to\cC$ with $J\in\cG$ is the maximum $\Gamma$-size of its bags $d(v)$ minus one:
	\[ w_\Gamma(d)\coloneqq \max_{v\in V(J)} s_\Gamma(d(v)) -1. \]
	Now, the \textit{$\Gamma$-width}~$w_\Gamma(X)$ of an object $X$ of~$\cC$ is defined as the minimum $\Gamma$-widths of all structured decompositions $d\co \intJ\to\cC$ with $J\in\cG$ whose colimit is isomorphic to~$X$:
	\[ w_\Gamma(X)\coloneqq \min_{\substack{d\colon\!\! \int\! J\to\cC \textup{ with} \\ J\in\cG,\, \colim d \,\cong\, X}} w_\Gamma(d). \]
\end{defi}

\section{Temporalization of a spined sd-category}
\label{section:temporalization}

Aiming to temporalize spined sd-categories, the guiding idea is to combine the concept of spined sd-categories with the framework of persistent narratives.

\subsection{Combining spined sd-categories and persistent narratives}

Let $\timeT$ be a finite discrete time category, let $\timeS\subseteq \timeT$ be a sub-join-semilattice, and let $\tau\co \timeS\hookrightarrow \timeT$ denote the inclusion functor.
Furthermore, let $(\cC,\cG,\Omega)$ be a spined sd-category which additionally satisfies the following condition:
\begin{enumerate}[leftmargin=0.5in]
	\renewcommand{\labelenumi}{(T\theenumi)}
	\item\label{axiom:T1} The category $\cC$ and its subcategories $\Omega_n\subseteq\cC$ admit all pullbacks, the inclusion functors $\iota_n\co \Omega_n\hookrightarrow\cC$ preserve pullbacks, and each $\Omega_n$ is full in~$\cC$.
\end{enumerate}
In particular, post-composition with~$\iota_n$ provides a well-defined functor
	\[ \Pe(\timeT,\iota_n)\co \Pe(\timeT,\Omega_n)\to \Pe(\timeT,\cC), \]
which we call the \textit{covariant change-of-base functor}, and the same holds for the subcategory $\timeS\subseteq\timeT$ in place of~$\timeT$.
Similarly, we have a \textit{change-of-temporal-resolution functor} defined by pre-composition with the opposite $\tau^\op\co \timeS^\op\hookrightarrow \timeT^\op$ of~$\tau$:
	\[ \Pe(\tau,\cC)\co \Pe(\timeT,\cC)\to \Pe(\timeS,\cC). \]
The same is true when replacing the category~$\cC$ by $\Omega_n$ for $n\geq0$.
Moreover, since the subcategories $\Omega_n\subseteq\cC$ are full, we observe that the inclusion functors $\iota_n\co\Omega_n\hookrightarrow\cC$ reflect pullbacks, meaning that any commutative square in~$\Omega_n$ that is a pullback square in~$\cC$ is also a pullback square in~$\Omega_n$.

\bigskip
Given a non-negative integer~$n$, we consider the following pullback square of categories and functors (adapted from \cite[Definition~2.20]{Bumpus-et-al_Time-Varying}): 

\begin{equation*}
\begin{tikzcd}
	\Pe(\timeT,\cC)\times_{\Pe(\timeS,\cC)}\Pe(\timeS,\Omega_n) \arrow[r] \arrow[d, "\pi_n", swap]
		\arrow[dr, phantom, "\usebox\pullbacksquare", very near start, color=black]
		& \Pe(\timeS,\Omega_n) \arrow[d, "{\Pe(\timeS,\iota_n)}"]
	\\ \Pe(\timeT,\cC) \arrow[r, "{\Pe(\tau,\cC)}", swap]
		& \Pe(\timeS,\cC).
\end{tikzcd}
\end{equation*}
We will regard the pullback category
$\cP_n\coloneqq \Pe(\timeT,\cC)\times_{\Pe(\timeS,\cC)}\Pe(\timeS,\Omega_n)$
as the subcategory of the usual product category $\Pe(\timeT,\cC)\times\Pe(\timeS,\Omega_n)$ whose objects and morphisms are given by the standard pullbacks of (possibly large) sets and set maps. 
Specifically, the objects of~$\cP_n$ are those pairs $(F,G_n)$ of persistent narratives $F\co \timeT^\op\to\cC$ and $G_n\co \timeS^\op\to\Omega_n$ such that the composite functors
	\[ F\tau^\op\co \timeS^\op\hookrightarrow \timeT^\op\xrightarrow{F}\cC \qquad \textup{and} \qquad \iota_n G_n\co \timeS^\op\xrightarrow{G_n}\Omega_n\hookrightarrow\cC \]
coincide, that is, the square below commutes:
\begin{equation*}
\begin{tikzcd}
	\timeS^\op \arrow[r, "G_n"] \arrow[d, "\tau^\op", hook, swap]
		& \Omega_n \arrow[d, "\iota_n", hook]
	\\
	\timeT^\op \arrow[r, "F", swap]
		& \cC.
\end{tikzcd}
\end{equation*}
Similarly, the morphisms of~$\cP_n$ from $(F,G_n)$ to $(F',G_n')$ are those pairs $(\varphi,\psi_n)$ of natural transformations $\varphi\co F\Rightarrow F'\co \timeT^\op\to\cC$ and $\psi_n\co G_n\Rightarrow G_n'\co \timeS^\op\to\Omega_n$ whose component morphisms $\varphi_{[a,b]}\co F([a,b])\to F'([a,b])$ and $(\psi_n)_{[a,b]}\co G_n([a,b])\to G_n'([a,b])$ in~$\cC$ agree for each time interval $[a,b]$ in~$\timeS$.
The functor $\pi_n\co\cP_n\to\Pe(\timeT,\cC)$ is given by canonical projection onto the first component, that is, $\pi_n(F,G_n)\coloneqq F$ and $\pi_n(\varphi,\psi_n)\coloneqq \varphi$. Note that the functor $\Pe(\timeS,\iota_n)$ is monic, and therefore, so is its pullback~$\pi_n$.
In fact, given $(F,G_n)\in\Ob(\cP_n)$, the narrative $G_n\co\timeS^\op\to\Omega_n$ is uniquely determined by $F\co\timeT^\op\to\cC$, and given $(\varphi,\psi_n)\co(F,G_n)\to(F',G_n')$ in~$\cP_n$, the component morphisms of $\psi_n\co G_n\Rightarrow G_n'$ in~$\cC$ are those of $\varphi\co F\Rightarrow F'$ for $[a,b]\in\Ob(\timeS)$.
We define~$\widehat{\Omega}_n$ to be the image category of $\cP_n$ under $\pi_n$:
\[ \widehat{\Omega}_n\coloneqq \pi_n(\cP_n)\subseteq\Pe(\timeT,\cC). \]

This subcategory is concretely given as follows. Its objects are those persistent narratives $F\co\timeT^\op\to\cC$ for which there exists a persistent narrative $G_n\co \timeS^\op\to\Omega_n$ such that the pair $(F,G_n)$ is an object of~$\cP_n$, which means that the restricted functors 
	$F\tau^\op\co \timeS^\op\hookrightarrow \timeT^\op\xrightarrow{F}\cC$ and $\iota_n G_n\co \timeS^\op\xrightarrow{G_n}\Omega_n\hookrightarrow\cC$
are equal. Moreover, a morphism in~$\widehat{\Omega}_n$ from $F=\pi_n(F,G_n)$ to~$F'=\pi_n(F,G'_n)$ is a natural transformation $\varphi\co F\Rightarrow F'\co \timeT^\op\to\cC$ for which there exists a natural transformation $\psi_n\co G_n\Rightarrow G'_n\co \timeS^\op\to\Omega_n$ such that the pair $(\varphi,\psi_n)$ is a morphism of~$\cP_n$, which means that for every $[a,b]\in\Ob(\timeS)$, we have $\varphi_{[a,b]} = (\psi_n)_{[a,b]}$, as morphisms from $F([a,b])=G_n([a,b])$ to $F'([a,b])=G'_n([a,b])$ in~$\cC$.
Since the functor $\pi_n\co \cP_n\to\Pe(\timeT,\cC)$ is monic, the subcategory $\widehat{\Omega}_n = \pi_n(\cP_n)$ of $\Pe(\timeT,\cC)$ is well-defined.

In the lemma below, we record an easier description of~$\widehat{\Omega}_n$, which substantially exploits the assumption that the subcategory $\Omega_n\subseteq\cC$ is full. We will use this result in the proof of our main theorem, Theorem~\ref{thm:temporalization-of-spined-sd-category}, several times.

\begin{lemma}\label{lemma:widehat-Omega-n}
	For every integer $n\geq0$, $\widehat{\Omega}_n$~is the full subcategory of $\Pe(\timeT,\cC)$ consisting of those persistent narratives $F\co\timeT^\op\to\cC$ whose value $F([a,b])$ at any time interval $[a,b]$ in~$\timeS$ is contained in~$\Omega_n$.
\end{lemma}

\begin{proof}
	It is obvious from the definition that every object of~$\widehat{\Omega}_n$ is a $\cC$-valued persistent narrative on~$\timeT$ sending any time interval in~$\timeS$ to an object of~$\Omega_n$. 
	
	Conversely, let $F\co\timeT^\op\to\cC$ be a persistent narrative with $F([a,b])\in\Ob(\Omega_n)$ for all $[a,b]\in\Ob(\timeS)$. Then for every morphism $i\co[a,b]\hookrightarrow[c,d]$ in~$\timeS$, $F(i)\co F([c,d])\to F([a,b])$ is a morphism in~$\cC$ between objects in~$\Omega_n$. Since the subcategory $\Omega_n\subseteq\cC$ is full, $F(i)$~is already a morphism in~$\Omega_n$.
	Thus, the restricted presheaf $F\tau^\op\co\timeS^\op\to\cC$~factors through the inclusion map $\iota_n\co\Omega_n\hookrightarrow\cC$, that is, there exists a presheaf $G_n\co\timeS^\op\to\Omega_n$ making the square below commute:
	\begin{equation*}
		\begin{tikzcd}
			\timeS^\op \arrow[r, "G_n", dashed] \arrow[d, hook, "\tau^\op", swap] 
			& \Omega_n \arrow[d, hook, "\iota_n"]
			\\
			\timeT^\op \arrow[r, "F", swap]
			& \cC.
		\end{tikzcd}
	\end{equation*}
	Since $F$ is supposed to be a persistent narrative, so is $\iota_nG_n=F\tau^\op\co\timeS^\op\to\cC$.
	Moreover, the full inclusion functor $\iota_n\co\Omega_n\hookrightarrow\cC$ reflects pullbacks, which implies that the presheaf $G_n\co\timeS^\op\to\Omega_n$ is a persistent narrative, as well, so $F$ is an object of~$\widehat{\Omega}_n$.
	
	To show that the subcategory $\widehat{\Omega}_n\subseteq\Pe(\timeT,\cC)$ is full, let $F,F'\co\timeT^\op\to\cC$ be objects of~$\widehat{\Omega}_n$, and let $\varphi$ be any morphism in $\Pe(\timeT,\cC)$ from $F$ to~$F'$, that is, any natural transformation $\varphi\co F\Rightarrow F'\co\timeT^\op\to\cC$.
	Then for every interval $[a,b]$ in~$\timeS$, $\varphi_{[a,b]}\co F([a,b])\to F'([a,b])$ is a morphism in~$\cC$ between objects in the full subcategory $\Omega_n\subseteq\cC$ and is therefore contained in~$\Omega_n$.
	Hence, writing $G_n,G_n'\co \timeS^\op\to\Omega_n$ for the narratives satisfying $F\tau^\op=\iota_nG$ and $F'\tau^\op=\iota_nG_n'$, respectively, then we have a natural transformation $\varphi\tau^\op\co \iota_nG_n\Rightarrow\iota_nG_n'\co\timeS^\op\to\cC$, whose component morphisms lie in~$\Omega_n$, so there exists a natural transformation $\psi_n\co G_n\Rightarrow G_n'\co\timeS^\op\to\Omega_n$ with $(\psi_n)_{[a,b]}=\varphi_{[a,b]}$ for all $[a,b]$ in~$\timeS$.
	Thus, the morphism $\varphi$ lies in~$\widehat{\Omega}_n$.
\end{proof}

\subsection{Temporalized sd-categories}

The next theorem provides a time-respecting notion of an sd-category. We will build upon it in our main theorem, establishing a temporalization of a spined sd-category, see Theorem~\ref{thm:temporalization-of-spined-sd-category}.

\begin{thm}\label{thm:termporalized-sd-category}
	Let again $\timeT$ be a finite discrete time category, let $\tau\co\timeS\hookrightarrow\timeT$ be the inclusion of a sub-join-semilattice  $\timeS\subseteq\timeT$, and let $(\cC,\cG,\Omega)$ be a spined sd-category satisfying~(T\ref{axiom:T1}). Furthermore, suppose that the following conditions hold:
	\begin{enumerate}[leftmargin=0.5in]
	\renewcommand{\labelenumi}{(T\theenumi)}
	\setcounter{enumi}{1}
		\item\label{axiom:T2} The category~$\cC$ is finitely cocomplete, that is, it admits all finite colimits.
		
		\item\label{axiom:T3} The inclusion functor $\Pe(\timeT,\cC)\hookrightarrow[\timeT^\op,\cC]$ admits a left adjoint
			\[ \mathfrak{S}\co[\timeT^\op,\cC]\to\Pe(\timeT,\cC) \]
		such that the composition $\Pe(\timeT,\cC)\hookrightarrow[\timeT^\op,\cC]\xrightarrow{\mathfrak{S}}\Pe(\timeT,\cC)$ is the identity functor.
	\end{enumerate}
	Then the pair $(\Pe(\timeT,\cC),\cG)$ is an sd-category.
\end{thm}

\begin{rmk}
	By assumption, $\mathfrak{S}$~is both a left adjoint and a retraction of the inclusion functor. We will therefore refer to it as the \textit{sheafification functor}. In fact, in Example~\ref{subsection:case-study-ordinary-tree-width}, where $\cC$ is the category $\Grph$ of graphs, $\mathfrak{S}$~is induced by the usual sheafification functor of presheaves with values in the category of sets.
\end{rmk}

\begin{proof}[Proof of Theorem~\ref{thm:termporalized-sd-category}]
	Let $J\in\cG$ be a graph, and let $d\co \intJ\to\cC$ be a $J$-structured decomposition in $\Pe(\timeT,\cC)$. We have to show that $d$ admits a colimit.
	For any interval $t=[a,b]\in\Ob(\timeT)$, we denote by $\textup{ev}_t\co[\timeT^\op,\cC]\to\cC$ the evaluation functor at~$t$ and consider the composite functor
	\[ \textstyle d_t\co \intJ\xrightarrow{d} \Pe(\timeT,\cC)\hookrightarrow[\timeT^\op,\cC]\xrightarrow{\textup{ev}_t} \cC, \]
	which sends an object $x\in\Ob(\intJ)$ to the object $d_t(x)=d(x)(t)\in\Ob(\cC)$ and a morphism $k\co x\to y$ in $\intJ$ to the morphism $d_t(k)=d(k)_t\co d(x)(t)\to d(y)(t)$ in~$\cC$.
	By definition, $d(k)\co d(x)\to d(y)$ is a monomorphism in $\Pe(\timeT,\cC)$. As a right adjoint, the inclusion functor $\Pe(\timeT,\cC)\hookrightarrow[\timeT^\op,\cC]$ preserves pullbacks and thus monomorphisms, so $d(k)$ is also a monomorphism in $[\timeT^\op,\cC]$.
	Since the category~$\cC$ admits all pullbacks, the evaluation functor $\textup{ev}_t\co[\timeT^\op,\cC]\to\cC$, as well, preserves pullbacks and thus monomorphisms, see \cite[Corollary 2.15.3]{Borceux_Handbook-I}, for instance. Hence, the $t$-component $d_t(k)=d(k)_t$ is a monomorphism in~$\cC$,
		% See also https://math.stackexchange.com/questions/3837806/monomorphisms-in-functor-categories where one can find moreover a counter-example.
	which shows that $d_t\co \intJ\to \cC$ is a $J$-structured decomposition in~$\cC$.
	By the defining property of the sd-category $(\cC,\cG)$, it follows that $d_t$ admits a colimit
	$\sBigl( \colim d_t,\, \lambda^t=(\lambda_x^t\co d(x)(t)\to\colim d_t)_{x\in\intJ} \sBigr)$.
	This provides a presheaf
	\[ \colim\,\!^\textup{pre}d\co \timeT^\op\to\cC \]
	taking any interval $t\in\Ob(\timeT)$ to the object $(\colim^\textup{pre}d)(t)\coloneqq \colim d_t\in\Ob(\cC)$ and any inclusion $i\co t\hookrightarrow t'$ of intervals in~$\timeT$ to the unique morphism $(\colim^\textup{pre}d)(i)\co \colim d_{t'}\to \colim d_t$ in~$\cC$ making the square below commute for any $x\in\Ob(\intJ)$:
	\begin{equation*}
		\begin{tikzcd}
			d(x)(t') \arrow[r, "\lambda_x^{t'}"] \arrow[d, "d(x)(i)", swap]
			& \colim d_{t'} \arrow[d, "(\colim^\textup{pre}d)(i)"]
			\\
			d(x)(t) \arrow[r, "\lambda_x^t", swap]
			& \colim d_t,
		\end{tikzcd}
	\end{equation*}
	which is obtained by applying the universal property of the colimit $(\colim d_{t'}, \lambda^{t'})$ to the cocone $\sBigl( \colim d_t, (\lambda_x^t\comp d(x)(i))_{x\in\intJ} \sBigr)$ under $d_{t'}\co \intJ\to\cC$.
	By construction, for each $x\in\Ob(\intJ)$, the morphisms $\lambda_x^t\co d(x)(t)\to\colim d_t$ with $t\in\Ob(\timeT)$ assemble into a natural transformation
	$\lambda_x\co d(x)\Rightarrow \colim^\textup{pre}d\co \timeT^\op\to\cC$
	with component morphisms $(\lambda_x)_t\coloneqq \lambda_x^t$, so $\lambda_x$ is a morphism in the presheaf category $[\timeT^\op,\cC]$.
	Their family $\lambda\coloneqq(\lambda_x)_{x\in\intJ}$ makes the presheaf $\colim^\textup{pre}$ into a cocone $(\colim^\textup{pre}d,\lambda)$ under the functor $\intJ\xrightarrow{d} \Pe(\timeT,\cC)\hookrightarrow[\timeT^\op,\cC]$.
	The sheafification functor $\mathfrak{S}\co [\timeT^\op,\cC]\to\Pe(\timeT,\cC)$ takes this cocone to a cocone $(\mathfrak{S}(\colim^\textup{pre}d),\mathfrak{S}(\lambda))$ under $\intJ\xrightarrow{d} \Pe(\timeT,\cC)\hookrightarrow[\timeT^\op,\cC]\xrightarrow{\mathfrak{S}} \Pe(\timeT,\cC)$, which is precisely the given functor $d\co\intJ\to\Pe(\timeT,\cC)$, where $\mathfrak{S}(\lambda)$ denotes the family of the morphisms $\mathfrak{S}(\lambda_x)\co d(x)\to\mathfrak{S}(\colim^\textup{pre})$ in $\Pe(\timeT,\cC)$ for $x\in\Ob(\intJ)$.
	
	We show that this cocone is a colimit of~$d$.
	Since the category~$\cC$ admits all finite colimits, the functor category $[\timeT^\op,\cC]$ also admits all finite colimits, and these may be computed pointwise. This follows from the dual result of \cite[Theorem 2.15.2]{Borceux_Handbook-I}.
	By assumption, the graph~$J$ is finite, so its barycentric subdivision~$\intJ$ is a finite category.
	The cocone $(\colim^\textup{pre}d,\lambda)$ under the functor $\intJ\xrightarrow{d}\Pe(\timeT,\cC)\hookrightarrow[\timeT^\op,\cC]$ was constructed pointwise from the colimits $(\colim d_t,\lambda^t)$ of $\intJ\xrightarrow{d}\Pe(\timeT,\cC)\hookrightarrow[\timeT^\op,\cC]\xrightarrow{\textup{ev}_t}\cC$ for $t\in\Ob(\timeT)$, which implies that $(\colim^\textup{pre}d,\lambda)$ is a colimit of $\intJ\xrightarrow{d}\Pe(\timeT,\cC)\hookrightarrow [\timeT^\op,\cC]$. As a left adjoint, the sheafification functor~$\mathfrak{S}$ preserves colimits, so we deduce that the cocone $(\mathfrak{S}(\colim^\textup{pre}d),\mathfrak{S}(\lambda))$under $d\co\intJ\to\cC$ is a colimit of~$d$.
	This shows that $(\Pe(\timeT,\cC),\cG)$ is an sd-category.
\end{proof}

\subsection{Main result}

Here is the main result of our paper.

\begin{thm}\label{thm:temporalization-of-spined-sd-category}
	As before, let $\timeT$ be a finite discrete time category, and let $\tau\co\timeS\hookrightarrow\timeT$ be the inclusion of a sub-join-semilattice $\timeS\subseteq\timeT$.
	Let $(\cC,\cG,\Omega)$ be a spined sd-category that satisfies the axioms (T\ref{axiom:T1}), (T\ref{axiom:T2}) and (T\ref{axiom:T3}) as well as 
	\begin{enumerate}[leftmargin=0.5in]
	\renewcommand{\labelenumi}{(T\theenumi)}
	\setcounter{enumi}{3}
		\item\label{axiom:T4} The category~$\cC$ admits pushout squares along monomorphisms and these are also pullback squares. Moreover, monomorphisms are stable under pushouts.
	\end{enumerate}
	Then the subcategories
		\[ \widehat{\Omega}_n = \pi_n(\cP_n)\subseteq\Pe(\timeT,\cC) \]
	with $n\geq0$ define a spine $\widehat{\Omega}$ on $\Pe(\timeT,\cC)$, thereby exhibiting the sd-category $(\Pe(\timeT,\cC),\cG)$ of Theorem~\ref{thm:termporalized-sd-category} as a spined sd-category $(\Pe(\timeT,\cC),\cG,\widehat{\Omega})$.
\end{thm}

Put into words, the theorem allows us to lift measurements of complexity from the static to the time-varying setting.

Note that the condition~(T\ref{axiom:T4}) is satisfied, for example, if $\cC$ is an adhesive category (\cite[Definition~5]{Lack-Sobocinksi_adhesive}). Indeed, the definition of an adhesive category includes the existence of pushouts along monomorphisms; by \cite[Lemma~13]{Lack-Sobocinksi_adhesive}, pushout squares along monomorphisms are also pullback squares; and by \cite[Lemma~12]{Lack-Sobocinksi_adhesive}, pushouts of monomorphisms are again monomorphisms.

We will split the proof of the theorem into several steps, each of them dedicated to the verification of one of the axioms (S\ref{axiom:S2}) to~(S\ref{axiom:S5}).
The axiom~(S\ref{axiom:S4}) turns out to be most demanding. The idea to obtain an object $W\co \timeT^\op\to\cC$ of~$\widehat{\Omega}_n$ as in~(S\ref{axiom:S4}) is to construct it on each interval of length at most one together with connecting monomorphisms and then to extend it to all of the time category~$\timeT$ by taking pullbacks. The first problem will be to establish coherent values for~$W$ on the intervals $[a,b]$ with $b-a\leq1$. This is where the axiom~(S\ref{axiom:T4}) comes into play, which allows us to form pushout squares along monomorphisms. Moreover, in this step, the finiteness of~$\timeT$ will be important so that we may take the maximum of the indices~$n$ of $\Omega_n$ with $W([a,b])\in\Omega_n$ over the chosen intervals $[a,b]\in\Ob(\timeT)$. The second problem then is to properly extend the individual monomorphisms. To do so, we apply the pasting law for pullback squares together with the fact that monomorphisms are stable under pullbacks.

\paragraph{Axiom~(S\ref{axiom:S2}).}
	For every non-negative integer~$n$, we have $\widehat{\Omega}_n\subseteq \widehat{\Omega}_{n+1}$ as subcategories of $\Pe(\timeT,\cC)$.
	
	\begin{proof}
		Since $\Omega = (\Omega_n)_{n\geq0}$ is a spine on~$\cC$, $\Omega_n$ is a subcategory of~$\Omega_{n+1}$, so the statement is an immediate consequence of Lemma~\ref{lemma:widehat-Omega-n}.
	\end{proof}	

We thus obtain a well-defined subcategory $\bigcup\widehat{\Omega}\coloneqq \bigcup_{n\geq0}\widehat{\Omega}_n$ of $\Pe(\timeT,\cC)$.

\paragraph{Axiom~(S\ref{axiom:S3}).}
	The subcategories $\bigcup\widehat{\Omega}\subseteq\Pe(\timeT,\cC)$ and $\Omega_n\subseteq\bigcup\widehat{\Omega}$ are closed under isomorphic objects.
	
	\begin{proof}
		We will show the stronger result that $\widehat{\Omega}_n$ is closed under isomorphic objects in $\Pe(\timeT,\cC)$.
		(Beware that this is not in general true for the spined sd-category $(\cC,\cG,\Omega)$, that is, $\Omega_n$ is not necessarily closed under isomorphic objects in~$\cC$, unless $\bigcup\Omega$ is full in~$\cC$.)
		Given objects $F$ of~$\widehat{\Omega}_n$ and $H$~of $\Pe(\timeT,\cC)$ together with an isomorphism $\varphi\co F\xRightarrow{\sim} H$ in $\Pe(\timeT,\cC)$, we need to show that $H$ is already contained in~$\widehat{\Omega}_n$.
		For each time interval $[a,b]$ in~$\timeS$, $\varphi_{[a,b]}\co F([a,b])\to H([a,b])$ is an isomorphism in~$\cC$ between the object $F([a,b])$ of $\Omega_n\subseteq \bigcup\Omega\coloneqq \bigcup_{m\geq0}\Omega_m$ and $H([a,b])$ of~$\cC$. Since $\Omega$ is a spine on~$\cC$, $\bigcup\Omega$~is closed under isomorphic objects as a subcategory of~$\cC$, so $H([a,b])$ lies in~$\Omega_m$ for some $m\geq0$. Thus, $\varphi_{[a,b]}$~is an isomorphism in~$\cC$ joining objects in $\Omega_n$ and~$\Omega_m$. Since these subcategories of~$\cC$ are full, $\varphi_{[a,b]}$~is also an isomorphism in~$\bigcup\Omega$. The assumption that $\Omega$ is a spine on~$\cC$ also guarantees that $\Omega_n$ is closed under isomorphic objects as a subcategory of $\bigcup\Omega$, so in view of the isomorphism $\varphi_{[a,b]}\co F([a,b])\xrightarrow{\sim} H([a,b])$ in~$\bigcup\Omega$ with $F([a,b])\in\Ob(\Omega_n)$ and $H([a,b])\in\Ob(\Omega_m)\subseteq\Ob(\bigcup\Omega)$, it follows that $H([a,b])$ is already contained in~$\Omega_n$.
		By virtue of the description of~$\widehat{\Omega}_n$ exposed in Lemma~\ref{lemma:widehat-Omega-n}, we may deduce that $H$ is an object of~$\widehat{\Omega}_n$. Hence, $\widehat{\Omega}_n$~is closed under isomorphic objects in $\Pe(\timeT,\cC)$.
	\end{proof}

\paragraph{Axiom~(S\ref{axiom:S4}).}
	Every persistent narrative $H\co \timeT^\op\to\cC$ admits a monomorphism $\varphi\co H\rightarrowtail F$ in $\Pe(\timeT,\cC)$ to a persistent narrative $F\co \timeT^\op\to\cC$ that is contained in $\widehat{\Omega}_n$ for some $n\geq0$.
	
	\begin{proof}
		As a spine on~$\cC$, $\Omega$~satisfies the analogous axiom~(S\ref{axiom:S4}), so given any interval $[a,a+1]\in\Ob(\timeT)$ of length one, there exists an object $G([a,a+1])\in\Omega_{n([a,a+1])}$ for some $n([a,a+1])\geq0$ together with a monomorphism $\varphi_{[a,a+1]}\co H([a,a+1])\rightarrowtail G([a,a+1])$ in~$\cC$.
		Moreover, given any span in~$\timeT$ of the form
		$[a,a+1]\xhookleftarrow{i}[a+1,a+1]\xhookrightarrow{j}[a+1,a+2]$,
		we have the following diagram in~$\cC$:
		\begin{equation*}
		\begin{tikzcd}
			H([a,a+1]) \arrow[r, "H(i)"] \arrow[d, "\varphi_{[a,a+1]}", rightarrowtail, swap]
				& H([a+1,a+1])
				& H([a+1,a+2]) \arrow[l, "H(j)", swap] \arrow[d, "\varphi_{[a+1,a+2]}", rightarrowtail]
			\\
			G([a,a+1]) & & G([a+1,a+2]).
		\end{tikzcd}
		\end{equation*}
		By the axiom~(T\ref{axiom:T4}), since $\varphi_{[a,a+1]}$ and $\varphi_{[a+1,a+2]}$ are monomorphisms, we may form pushout squares as shown below:
		\begin{equation*}
		\begin{tikzcd}
			H([a,a+1]) \arrow[rr, "H(i)"] \arrow[d, "\varphi_{[a,a+1]}", rightarrowtail, swap]
				& & H([a+1,a+1]) \arrow[ld] \arrow[rd]
				& & H([a+1,a+2]) \arrow[ll, "H(j)", swap] \arrow[d, "\varphi_{[a+1,a+2]}", rightarrowtail]
			\\
			G([a,a+1]) \arrow[r] & G_{a+1}' & & G_{a+1}'' & G([a+1,a+2]). \arrow[l]
			\arrow["\mathlarger{\lrcorner}"{anchor=center, pos=0.125, rotate=180}, draw=none, from=2-2, to=1-1]
			\arrow["\mathlarger{\lrcorner}"{anchor=center, pos=0.125, rotate=90}, draw=none, from=2-4, to=1-5]
		\end{tikzcd}
		\end{equation*}		
		Again by~(T\ref{axiom:T4}), the morphisms from $H([a+1,a+1])$ to $G_{a+1}'$ and~$G_{a+1}''$, respectively, are monic, so (T\ref{axiom:T4}) allows us to form a further pushout square in~$\cC$, which by~(T\ref{axiom:T4}) is also a pullback square and consists entirely of monomorphisms:
		\begin{equation*}
		\begin{tikzcd}
				& H([a+1,a+1]) \arrow[ld, rightarrowtail, shorten <=0.2cm] \arrow[rd, rightarrowtail, shorten <=0.2cm]
			\\ 
			G_{a+1}' \arrow[rd, rightarrowtail] & & G_{a+1}'' \arrow[ld, rightarrowtail]
			\\
				& G_{a+1}'''.
			\arrow["\mathlarger{\lrcorner}"{anchor=center, pos=0.125, rotate=315}, draw=none, from=1-2, to=3-2]
			\arrow["\mathlarger{\lrcorner}"{anchor=center, pos=0.125, rotate=135}, draw=none, from=3-2, to=1-2]
		\end{tikzcd}
		\end{equation*}
		In virtue of the axiom~(S\ref{axiom:S4}) for the spine $\Omega$ on~$\cC$, we can choose a monomorphism in~$\cC$ from $G_{a+1}'''$ to an object $G([a+1,a+1])$ of $\Omega_{n([a+1,a+1])}$ for some $n([a+1,a+1])\geq0$.
		Define $\varphi_{[a+1,a+1]}$ to be the composite monomorphism
		\begin{equation*}
			\varphi_{[a+1,a+1]}\co H([a+1,a+1])\rightarrowtail G_{a+1}' \rightarrowtail G_{a+1}''' \rightarrowtail G([a+1,a+1]),
		\end{equation*}
		and define $G(i)$ and $G(j)$ to be the composite morphisms making the following diagram in~$\cC$ commute:
		\begin{equation*}
		\begin{tikzcd}[row sep=large]
			H([a,a+1]) \arrow[rr, "H(i)"] \arrow[d, "\varphi_{[a,a+1]}", rightarrowtail, swap]
				& & H([a+1,a+1]) \arrow[ld, rightarrowtail, shorten <=0.2cm] \arrow[rd, rightarrowtail, shorten <=0.2cm]
				& & H([a+1,a+2]) \arrow[ll, "H(j)", swap] \arrow[d, "\varphi_{[a+1,a+2]}", rightarrowtail]
			\\
			G([a,a+1]) \arrow[r] \arrow[rrd, "G(i)", shorten >=0.2cm, swap]
				& G_{a+1}' \arrow[r, rightarrowtail]
				& G_{a+1}''' \arrow[d, rightarrowtail]
				& G_{a+1}'' \arrow[l, rightarrowtail] 
				& G([a+1,a+2]) \arrow[l] \arrow[lld, "G(j)", shorten >=0.2cm]
			\\
				& & G([a+1,a+1]).
			\arrow["\mathlarger{\lrcorner}"{anchor=center, pos=0.125, rotate=0}, draw=none, from=1-1, to=2-2]
			\arrow["\mathlarger{\lrcorner}"{anchor=center, pos=0.125, rotate=180}, draw=none, from=2-2, to=1-1]
			\arrow["\mathlarger{\lrcorner}"{anchor=center, pos=0.125, rotate=270}, draw=none, from=1-5, to=2-4]
			\arrow["\mathlarger{\lrcorner}"{anchor=center, pos=0.125, rotate=90}, draw=none, from=2-4, to=1-5]
			\arrow["\mathlarger{\lrcorner}"{anchor=center, pos=0.125, rotate=315}, draw=none, from=1-3, to=2-3]
			\arrow["\mathlarger{\lrcorner}"{anchor=center, pos=0.125, rotate=135}, draw=none, from=2-3, to=1-3]
		\end{tikzcd}
		\end{equation*}
		
		This yields the following commutative diagram in~$\cC$:
		\begin{equation}\label{eqn:resulting-diagram-of-two-squares}
			\begin{tikzcd}
				H([a,a+1]) \arrow[r, "H(i)"] \arrow[d, "\varphi_{[a,a+1]}", rightarrowtail, swap]
				& H([a+1,a+1]) \arrow[d, "\varphi_{[a+1,a+1]}", rightarrowtail]
				& H([a+1,a+2]) \arrow[l, "H(j)", swap] \arrow[d, "\varphi_{[a+1,a+2]}", rightarrowtail]
				\\
				G([a,a+1]) \arrow[r] 
				& G([a+1,a+1])
				& G([a+1,a+2]) \arrow[l]
			\end{tikzcd}
		\end{equation}
		
		We observe that both squares are pullback squares. For the left-hand square, this follows from the fact that 
		\begin{equation*}
			\begin{tikzcd}
				H([a,a+1]) \arrow[r, "H(i)"] \arrow[d, "\varphi_{[a,a+1]}", rightarrowtail, swap]
				& H([a+1,a+1]) \arrow[d, "\varphi_{[a+1,a+1]}", rightarrowtail]
				\\
				G([a,a+1]) \arrow[r, "G(i)", swap] 
				& G([a+1,a+1])
			\end{tikzcd}
		\end{equation*}
		is a pullback square and that the composite morphism
		\[ G_{a+1}'\rightarrowtail G_{a+1}'''\rightarrowtail G([a+1,a+1]) \]
		is monic. Indeed, to verify the universal property of a pullback square, suppose we are given any object $A\in\Ob(\cC)$ and any two morphisms $f\co A\to G([a,a+1])$ and $g\co A\to H([a+1,a+1])$ with $G(i)f=\varphi_{[a+1,a+1]}g$.
		By definition of the morphisms $G(i)$ and $(\varphi_{[a+1,a+1]})$, this means that the diagram below commutes:
		\begin{equation*}
		\begin{tikzcd}
			A \arrow[r, "g"] \arrow[d, "f", swap]
				& H([a+1,a+1]) \arrow[r, rightarrowtail]
				& G_{a+1}' \arrow[d, rightarrowtail]
			\\
			G([a,a+1]) \arrow[d]
				& & G_{a+1}''' \arrow[d, rightarrowtail]
			\\
			G_{a+1}' \arrow[r, rightarrowtail]
				& G_{a+1}''' \arrow[r, rightarrowtail]
				& G([a+1,a+1])
		\end{tikzcd}
		\end{equation*}
		where the right-hand arrows and the bottom arrows are composed of the same monomorphisms, which implies commutativity of the square
		\begin{equation*}
		\begin{tikzcd}
			A \arrow[r, "g"] \arrow[d, "f", swap]
			& H([a+1,a+1]) \arrow[d, rightarrowtail]
			\\
			G([a,a+1]) \arrow[r]
			& G_{a+1}'.
		\end{tikzcd}
		\end{equation*}
		Recalling that
		\begin{equation*}
			\begin{tikzcd}
				H([a,a+1]) \arrow[r, "H(i)"] \arrow[d, "\varphi_{[a,a+1]}", swap]
				& H([a+1,a+1]) \arrow[d, rightarrowtail]
				\\
				G([a,a+1]) \arrow[r]
				& G_{a+1}'
			\end{tikzcd}
		\end{equation*}
		is a pullback square, we deduce that there exists a unique morphism $u\co A\to H([a,a+1])$ in~$\cC$ satisfying $\varphi_{[a,a+1]}u=f$ and $H(i)u=g$. Hence, the left-hand square of the diagram~(\ref{eqn:resulting-diagram-of-two-squares}) is indeed a pullback square, and the analogous argument applies to the square on the right-hand side.
		
		Let $\timeT_{\leq1}\subseteq \timeT$ denote the full subcategory containing as objects those intervals $[a,b]\in\Ob(\timeT)$ whose length $b-a$ is at most one.
		The assumption that $\timeT$ is finite guarantees that $\timeT_{\leq1}$ is a zigzag diagram, meaning that it has the following form for some non-negative integer~$N$, where we assume without loss of generality that $[0,0]\in\Ob(\timeT)$:
		\begin{equation*}
		\begin{tikzcd}[row sep=small]
			& {[0,1]} & & {[1,2]} & \dots & {[N-1,N]}
			\\ 
			{[0,0]} \arrow[ru, hook] & & {[1,1]} \arrow[lu, hook'] \arrow[ru, hook] & & \dots & & {[N,N]}. \arrow[lu, hook'] 
		\end{tikzcd}
		\end{equation*}
		We inductively apply the above procedure to obtain a diagram of pullback squares in~$\cC$ as depicted below, where in the first and last square, we can simply take $G([0,0])\coloneqq G_0''$ and $G([N,N])\coloneqq G_N'$, respectively:
		\begin{equation}\label{eqn:constr-mono-H-G}
	\begin{adjustbox}{max width=\displaywidth}			
		\begin{tikzcd}
				& H([0,1]) \arrow[ld] \arrow[dd, "{\varphi_{[0,1]}}"{description, pos=0.6}, rightarrowtail] 
				&[-1.2cm] &[-1.2cm] H([a,a+1]) \arrow[rd, shorten >=0.2cm] \arrow[dd, "{\varphi_{[a,a+1]}}"{description, pos=0.6}, rightarrowtail]
				& & H([a+1,a+2]) \arrow[ld, shorten >=0.2cm] \arrow[dd, "{\varphi_{[1+2,1+2]}}"{description, pos=0.6}, rightarrowtail] 
				&[-1.2cm] &[-1.2cm] H([N-1,N]) \arrow[rd] \arrow[dd, "{\varphi_{[N-1,N]}}"{description, pos=0.6}, rightarrowtail]
			\\
			H([0,0]) \arrow[dd, "{\varphi_{[0,0]}}"{description, pos=0.4}, rightarrowtail] 
				& & \quad \dots 
				& & H([a+1,a+1]) \arrow[dd, "{\varphi_{[a+1,a+1]}}"{description, pos=0.4}, rightarrowtail]  
				& & \quad \dots 
				& & H([N,N]) \arrow[dd, "{\varphi_{[N,N]}}"{description, pos=0.4}, rightarrowtail] 
			\\
				& G([0,1]) \arrow[ld] & & G([a,a+1]) \arrow[rd, shorten >=0.2cm] & & G([a+1,a+2]) \arrow[ld, shorten >=0.2cm] & & G([N-1,N]) \arrow[rd]
			\\
			G([0,0]) & & & & G([a+1,a+1]) & & & & G([N,N]).
			\arrow["\mathlarger{\lrcorner}"{anchor=center, pos=0.125, rotate=270}, draw=none, from=1-2, to=4-1]
			\arrow["\mathlarger{\lrcorner}"{anchor=center, pos=0.125, rotate=0}, draw=none, from=1-4, to=4-5]
			\arrow["\mathlarger{\lrcorner}"{anchor=center, pos=0.125, rotate=270}, draw=none, from=1-6, to=4-5]
			\arrow["\mathlarger{\lrcorner}"{anchor=center, pos=0.125, rotate=0}, draw=none, from=1-8, to=4-9]
		\end{tikzcd}
	\end{adjustbox}
		\end{equation}
		
		By construction, for each interval $[a,b]\in\Ob(\timeT_{\leq1})$ of length at most one, we have $G([a,b])\in\Omega_{n([a,b])}$ for some $n([a,b])\geq0$.
		Finiteness of~$\timeT$ allows us to take the maximum $n\geq0$ of these integers $n([a,b])$ over all $[a,b]\in\Ob(\timeT_{\geq1})$. Then for each such interval $[a,b]$, we have $G([a,b])\in\Omega_{n([a,b])}\subseteq\Omega_n$.
		Since the category~$\Omega_n$ admits all pullbacks, we may extend the lower row of the diagram~(\ref{eqn:constr-mono-H-G}) to a persistent narrative $G\co \timeT^\op\to\Omega_n$. (See also \cite[Proposition~2.9]{Bumpus-et-al_Time-Varying}.)
		Post-composition with the inclusion functor $\iota_n\co\Omega_n\hookrightarrow\cC$ provides a persistent narrative
		\[ F\coloneqq \Pe(\timeT,\iota_n)(G) = \iota_nG\co \timeT^\op\xrightarrow{G}\Omega_n\hookrightarrow\cC. \]
		It sends each interval $[a,b]$ in~$\timeS$ to an object in~$\Omega_n$, so $F$ is contained in $\widehat{\Omega}_n$.
		
		We want to extend the family of the above-constructed monomorphisms $\varphi_{[a,b]}\co H([a,b])\rightarrowtail G([a,b])=F([a,b])$ in~$\cC$ indexed by $[a,b]\in\Ob(\timeT_{\geq1})$ to a natural transformation $\varphi\co H\Rightarrow F\co \timeT^\op\to\cC$. For this, we will construct the missing monomorphisms $\varphi_{[a,b]}$ for the intervals $[a,b]\in\Ob(\timeT)$ of length $b-a>1$ inductively.
		Given any integer $a\in\{0,\dots,N-1\}$, since $F\co \timeT^\op\to\cC$ is a persistent narrative, the following is a pullback square in~$\cC$:
		\begin{equation}\label{eqn:pullback-square-F}
		\begin{tikzcd}
				& F([a,a+2]) \arrow[ld] \arrow[rd, shorten >=0.2cm]
			\\ 
			F([a,a+1]) \arrow[rd, shorten >=0.2cm] & & F([a+1,a+2]) \arrow[ld, shorten >=0.2cm]
			\\
				& F([a+1,a+1]).
			\arrow["\mathlarger{\lrcorner}"{anchor=center, pos=0.125, rotate=315}, draw=none, from=1-2, to=3-2]
		\end{tikzcd}
		\end{equation}
		By construction, it fits into the diagram below, whose top square commutes by functoriality of $H\co \timeT^\op\to\cC$ and whose two front squares commute by construction:
		\begin{equation}\label{eqn:towards-comm-cube-H-F}
			\begin{tikzcd}[row sep=large]
				& H([a,a+2]) \arrow[ld] \arrow[rd, shorten >=0.2cm]
				\\ 
				H([a,a+1]) \arrow[rd, shorten >=0.2cm] \arrow[d, "\varphi_{[a,a+1]}", rightarrowtail, swap]
					& F([a,a+2]) \arrow[ld] \arrow[rd, shorten >=0.2cm] 
					& H([a+1,a+2]) \arrow[ld, shorten >=0.2cm] \arrow[d, "\varphi_{[a+1,a+2]}", rightarrowtail]
				\\
				F([a,a+1]) \arrow[rd, shorten >=0.2cm] 
					& H([a+1,a+1]) \arrow[d, "\varphi_{[a+1,a+1]}", rightarrowtail]
					& F([a+1,a+2]) \arrow[ld, shorten >=0.2cm]
				\\
					& F([a+1,a+1]).
				\arrow["\mathlarger{\lrcorner}"{anchor=center, pos=0.125, rotate=315}, draw=none, from=2-2, to=4-2]
			\end{tikzcd}
		\end{equation}
		In particular, the ``outer'' hexagon
		\begin{equation*}
		\begin{tikzcd}
			& H([a,a+2]) \arrow[ld] \arrow[rd, shorten >=0.2cm]
			\\ 
			H([a,a+1]) \arrow[d, "\varphi_{[a,a+1]}", rightarrowtail, swap]
			& & H([a+1,a+2]) \arrow[d, "\varphi_{[a+1,a+2]}", rightarrowtail]
			\\
			F([a,a+1]) \arrow[rd, shorten >=0.2cm] 
			& & F([a+1,a+2]) \arrow[ld, shorten >=0.2cm]
			\\
			& F([a+1,a+1])
		\end{tikzcd}
		\end{equation*}
		is commutative.
		Thus, by the universal property of the pullback square~(\ref{eqn:pullback-square-F}), there exists a unique morphism
		\[ \varphi_{[a,a+2]}\co H([a,a+2])\to F([a,a+2]) \]
		in~$\cC$ making the following diagram commute:
		\begin{equation*}
		\begin{tikzcd}
			& H([a,a+2]) \arrow[ld] \arrow[rd, shorten >=0.2cm] \arrow[d, "\varphi_{[a,a+2]}"]
			\\ 
			H([a,a+1]) \arrow[d, "\varphi_{[a,a+1]}", rightarrowtail, swap]
			& F([a,a+2]) \arrow[ld] \arrow[rd, shorten >=0.2cm] 
			& H([a+1,a+2]) \arrow[d, "\varphi_{[a+1,a+2]}", rightarrowtail]
			\\
			F([a,a+1]) 
			& & F([a+1,a+2]).
		\end{tikzcd}
		\end{equation*}
		Hence, the diagram~(\ref{eqn:towards-comm-cube-H-F}) extends to a commutative cube in~$\cC$:
		\begin{equation}\label{eqn:comm-cube-H-F}
			\begin{tikzcd}[row sep=large]
					& H([a,a+2]) \arrow[ld] \arrow[rd, shorten >=0.2cm] \arrow[d, "\varphi_{[a,a+2]}"]
				\\ 
				H([a,a+1]) \arrow[rd, shorten >=0.2cm] \arrow[d, "\varphi_{[a,a+1]}", rightarrowtail, swap]
					& F([a,a+2]) \arrow[ld] \arrow[rd, shorten >=0.2cm] 
					& H([a+1,a+2]) \arrow[ld, shorten >=0.2cm] \arrow[d, "\varphi_{[a+1,a+2]}", rightarrowtail]
				\\
				F([a,a+1]) \arrow[rd, shorten >=0.2cm] 
					& H([a+1,a+1]) \arrow[d, "\varphi_{[a+1,a+1]}", rightarrowtail]
					& F([a+1,a+2]) \arrow[ld, shorten >=0.2cm]
				\\
					& F([a+1,a+1]).
				\arrow["\mathlarger{\lrcorner}"{anchor=center, pos=0.125, rotate=315}, draw=none, from=2-2, to=4-2]
			\end{tikzcd}
		\end{equation}
		Let us show that $\varphi_{[a,a+2]}$ is a monomorphism. For this purpose, we consider the following diagrams in~$\cC$:
		\begin{equation*}
		\begin{tikzcd}
			H([a,a+2]) \arrow[r] \arrow[d] \arrow[dr, phantom, "\textup{(I)}"]
				& H([a+1,a+2]) \arrow[d]
			\\
			H([a,a+1]) \arrow[r] \arrow[d, "\varphi_{[a,a+1]}", rightarrowtail, swap] \arrow[dr, phantom, "\textup{(II)}"]
				& H([a+1,a+1]) \arrow[d, "\varphi_{[a+1,a+1]}", rightarrowtail]
			\\
			F([a,a+1]) \arrow[r]
				& F([a+1,a+1])	
		\end{tikzcd}
		\qquad \qquad
		\begin{tikzcd}
			H([a,a+2]) \arrow[r] \arrow[d, "\varphi_{[a,a+2]}", swap] \arrow[dr, phantom, "\textup{(III)}"]
			& H([a+1,a+2]) \arrow[d, "\varphi_{[a+1,a+2]}", rightarrowtail]
			\\
			F([a,a+2]) \arrow[r] \arrow[d] \arrow[dr, phantom, "\textup{(IV)}"]
			& F([a+1,a+2]) \arrow[d]
			\\
			F([a,a+1]) \arrow[r]
			& F([a+1,a+1])
		\end{tikzcd}
		\end{equation*}
		where the square~(I) is the top square of the diagram~(\ref{eqn:comm-cube-H-F}), (II)~is the left-hand front square, (III)~is the right-hand back square, and (IV) is the bottom square of~(\ref{eqn:comm-cube-H-F}).
		Since $H\co \timeT^\op\to\cC$ is a persistent narrative, (I)~is a pullback square, and by construction, (II)~is a pullback square, as well. By the pasting law for pullback squares, this implies that so is their composite square, meaning the outer diagram of the diagram~$\frac{\textup{(I)}}{\textup{(II)}}$.
		By commutativity of the left-hand back square and the right-hand front square of the cube~(\ref{eqn:comm-cube-H-F}), this outer diagram coincides with the outer diagram of the diagram~$\frac{\textup{(III)}}{\textup{(IV)}}$, that is, the composite square of (III) and~(IV). Since (IV) is a pullback square, it follows, again by the pasting law for pullback squares, that also (III) is a pullback square.
		Thus, the morphism $\varphi_{[a,a+2]}\co H([a,a+2])\to F([a,a+2])$ is a pullback of the monomorphism $\varphi_{[a+1,a+2]}\co H([a+1,a+2])\rightarrowtail F([a+1,a+2])$, which shows that $\varphi_{[a,a+2]}$, as well, is a monomorphism.
		
		Inductively, proceeding in this manner, since $\timeT$ is finite, we obtain a family $\varphi\coloneqq (\varphi_{[a,b]})_{[a,b]\in\Ob(\timeT)}$ of monomorphisms $\varphi_{[a,b]}\co H([a,b])\rightarrowtail F([a,b])$ in~$\cC$ for each interval $[a,b]\in\Ob(\timeT)$ such that for every morphism $i\co [a,b]\hookrightarrow[c,d]$ in~$\timeT$, the square
		\begin{equation*}
		\begin{tikzcd}
			H([c,d]) \arrow[r, "H(i)"] \arrow[d, "\varphi_{[c,d]}", rightarrowtail, swap]
				& H([a,b]) \arrow[d, "\varphi_{[a,b]}", rightarrowtail]
			\\
			F([c,d]) \arrow[r, "F(i)", swap]
				& F([a,b])
		\end{tikzcd}
		\end{equation*}
		commutes. 
		Hence, $\varphi$ is a natural transformation $\varphi\co H\Rightarrow F\co \timeT^\op\to\cC$.
		This means that $\varphi$~is a morphism in the functor category $[\timeT^\op,\cC]$, and since $H$ and $F$ are persistent narratives, $\varphi$~is also a morphism in the full subcategory $\Pe(\timeT,\cC)\subseteq [\timeT^\op,\cC]$.
		Since each of its component morphisms is a monomorphism in~$\cC$, $\varphi$~is a monomorphism in $[\timeT^\op,\cC]$, so in particular, it is monic in~$\Pe(\timeT,\cC)$.
		
		We conclude that for any persistent narrative $H\co \timeT^\op\to\cC$, there exists a persistent narrative $F\co \timeT^\op\to\cC$ contained in~$\widehat{\Omega}_n$ for some integer $n\geq0$ together with a monomorphism $\varphi\co H\Rightarrow F$ in~$\Pe(\timeT,\cC)$, which completes the proof.
	\end{proof}

\paragraph{Axiom~(S\ref{axiom:S5}).}
	For every positive integer $n\geq1$, given persistent narratives $F\in\Ob(\widehat{\Omega}_n)$ and $\wt{F}\in\Ob(\widehat{\Omega}_{n-1})$ together with a monomorphism $\varphi\co F\rightarrowtail\wt{F}$ in $\Pe(\timeT,\cC)$, then $F$ is already contained in~$\widehat{\Omega}_{n-1}$.

	\begin{proof}
		The monomorphism $\varphi\co F\rightarrowtail\wt{F}$ in $\Pe(\timeT,\cC)$ is objectwise a monomorphism in~$\cC$.
		For each interval $[a,b]$ in~$\timeS$, we thus have a monomorphism $\varphi_{[a,b]}\co F([a,b])\rightarrowtail\wt{F}([a,b])$ in~$\cC$ between objects $F([a,b])$ and $\wt{F}([a,b])$ in~$\Omega_{n-1}$.
		By the axiom~(S\ref{axiom:S5}) for the spine~$\Omega$ on~$\cC$, it follows that $F([a,b])$ lies in~$\Omega_{n-1}$.
		Now, Lemma~\ref{lemma:widehat-Omega-n} implies that the persistent narrative $F\co \timeT^\op\to\cC$ is an object of~$\widehat{\Omega}_{n-1}$.
	\end{proof}

This completes the proof of Theorem~\ref{thm:temporalization-of-spined-sd-category}.

\begin{rmk}
	Hitherto, we have only considered finite simple undirected graphs. To accommodate more general settings, we might also allow for the indexing graph of a structured decomposition to be another kind of a graph, for example, allowing loops, multiple edges or directed edges. 
		% See also \cite[Footnote~5]{Bumpus-et-al_Structured-Decompositions}.
	Its barycentric subdivision is defined in the analogous way, and the definition of a structured decomposition remains the same.
	Moreover, we may also allow infinite graphs, meaning that the sets of vertices and edges may be infinite. In this case, however, to guarantee that Theorems \ref{thm:termporalized-sd-category} and~\ref{thm:temporalization-of-spined-sd-category} remain valid, we would need to strengthen the axiom~(T\ref{axiom:T2}), now requiring the existence of all small colimits in~$\cC$ rather than merely of all finite colimits. This is important for the proof of Theorem~\ref{thm:termporalized-sd-category}, which can be then carried out in the same way as before. The proof of Theorem~\ref{thm:temporalization-of-spined-sd-category}, building upon the result of~\ref{thm:termporalized-sd-category}, does not again use the finiteness of the indexing graph.
	Allowing infinite graphs would furthermore require us to replace the maximum and minimum in Definition~\ref{defi:width-Gamma-width} (width notions) by a supremum and infimum, respectively.
\end{rmk}

\section{Case studies: tree-width analogues for temporal graphs}
\label{section:case-studies}

Theorem~\ref{thm:temporalization-of-spined-sd-category} can be applied to many different contexts and provides temporal analogues of ordinary static spined sd-categories. As explained in Paragraph~\ref{paragraph:width-in-spined-sd-category}, spined sd-categories constitute a systematic framework to study generalized notions of tree-decompositions and tree-widths of classical graph theory, which is the reason why the temporalized spined sd-categories of Theorem~\ref{thm:temporalization-of-spined-sd-category} are expected to create a systematic setting to investigate time-varying generalizations of tree-decompositions and tree-widths. We will illustrate this by the following three case studies, resulting in time-dependent concepts of ordinary tree-width, complemented tree-width and the tree-independence number. The idea is to consider the static spined sd-categories of~\cite{Bumpus-et-al_Structured-Decompositions} that recover the time-independent notions of ordinary tree-width, complemented tree-width and the tree-independence number and to apply our theorem to those contexts. Although not obvious in the first place, we will show that, perhaps under slight modifications of the assumptions, this leads to sensible temporal analogues that recover their static counterparts, which also justifies that Theorem~\ref{thm:temporalization-of-spined-sd-category} indeed gives rise to a reasonable temporal generalization of the ordinary notion of a spined sd-category.

\subsection{Case study: Ordinary tree-width}
\label{subsection:case-study-ordinary-tree-width}

In this section, we want to temporalize the spined sd-category of \cite[Proposition 3.1.1]{Bumpus-et-al_Structured-Decompositions}, which consists of the category $\Grph$ of finite simple undirected graphs (having neither loops nor multiple edges), the index class $\{\textup{trees}\}$ and the spine whose $n$-th layer is the full subcategory spanned by the complete graphs of order at most~$n$.
By \cite[Proposition 3.1.1]{Bumpus-et-al_Structured-Decompositions}, this yields a spined sd-category~$\Gamma$, and the $\Gamma$-width of a graph is its classical tree-width.
In order to transfer this example to our setting, a naive idea would thus be to apply the temporalization method of Theorem~\ref{thm:temporalization-of-spined-sd-category} to this spined sd-category.
However, the problem is that the underlying category $\Grph$ of finite simple undirected graphs is not finitely cocomplete, that is, it violates axiom~(T\ref{axiom:T2}). Another problem arising with simple graphs is that we are now not only interested in graphs themselves but also in graph morphisms and that those are not allowed to smash edges into single vertices. (This will become relevant in the proof of Proposition~\ref{prop:case-study-ordinary-tree-width} below.)
To remedy these issues, we replace the category $\Grph$ by the category $\Grefl$ of finite reflexive graphs.
Here, a \textit{finite reflexive graph}, or simply, a \textit{reflexive graph}, is a pair $G=(V(G),E(G))$ consisting of a finite set~$V(G)$ and a set~$E(G)$ of one- or two-element subsets of~$V(G)$ such that $\{v\}$ is contained in $E(G)$ for each $v\in V(G)$. Using the ordinary graph theoretical terminology, reflexive graphs are precisely the finite undirected graphs with exactly one loop at each vertex but no multiple edges.
Given vertices $v$ and~$w$ of~$G$, we denote the edges of~$G$ again by~$vw$, meaning the set $\{v,w\}$ if $v$ and~$w$ are distinct and the singleton~$\{v\}$ if $v$ and~$w$ agree.
A \textit{morphism} $f\co G\to H$ of reflexive graphs is a map of vertex sets $f\co V(G)\to V(H)$ preserving edges, that is, $f$~takes any edge $vw$ in~$G$ to an edge $f(v)f(w)$ in~$H$.
Reflexive graphs and their morphisms form a category, which we denote by~$\Grefl$. By \cite[Appendix~A.1]{Bumpus-et-al_Structured-Decompositions}, this is a quasitopos.
Our index class~$\cG$ is still the class of all trees, and for each non-negative integer~$n$, we let $\Omega_n$ be the full subcategory of $\Grefl$ whose objects are the complete reflexive graphs of order at most~$n$. We put $\Omega\coloneqq (\Omega_n)_{n\geq0}$ and write $\iota_n\co \Omega_n\hookrightarrow\Grefl$ for the inclusion functors.

The same argument as in the proof of Proposition~3.1.1 in~\cite{Bumpus-et-al_Structured-Decompositions} shows that the triple $\Gamma\coloneqq (\Grefl, \{\textup{trees}\}, \Omega)$ is a spined sd-category.
By Theorem~\ref{thm:temporalization-of-spined-sd-category}, given any sub-join-semlilattice $\timeS\subseteq\timeT$ of a finite discrete time category~$\timeT$, $\Gamma$~induces a spined sd-category $\widehat{\Gamma}\coloneqq \sBigl( \Pe(\timeT, \Grefl), \{\textup{trees}\}, \widehat{\Omega} \sBigr)$.

Let us compute the $\Gamma$-size of a reflexive graph. (See also \cite[Proposition 3.1.1]{Bumpus-et-al_Structured-Decompositions}.)

\begin{lemma}
	A \textit{morphism} $f\co G\to H$ of reflexive graphs is a monomorphism in the category~$\Grefl$ if and only if it is injective as a map of vertex sets $f\co V(G)\to V(H)$.
\end{lemma}

\begin{proof}
	It is obvious that injectivity of $f\co V(G)\to V(H)$ is sufficient to guarantee that $f$ is a monomorphism in~$\Grefl$ since equality of graph morphisms is characterized pointwise. Conversely, let us suppose that $f$ is monic. Let $v,w\in V(G)$ be any vertices with $f(v)=f(w)$. We denote by $J$ the reflexive graph with exactly one looped vertex~$x$. Then there are graph morphisms $g,h\co J\to G$ defined by $g(x)\coloneqq v$ and $h(x)\coloneqq w$, respectively. They satisfy $fg=fh$, which implies that $g=h$ and therefore $v=w$.
\end{proof}

Thus, given any reflexive graph~$G$, its $\Gamma$-size $s_\Gamma(G)$ is the minimum non-negative integer~$n$ for which there exist a graph $K$ in~$\widehat{\Omega}_n$ and a graph morphism $f\co G\to K$ which is injective as a map $f\co V(G)\to V(K)$. Since $K$ is complete, every such map defines a graph morphism, and since the sets $V(G)$ and $V(K)$ are finite, existence of an injective map $f\co V(G)\to V(K)$ is equivalent to the condition $\abs{V(G)}\leq\abs{V(K)}$. Therefore, $s_\Gamma(G)$ is the minimum non-negative integer~$n$ such that there exists a complete reflexive graph~$K$ of order at most~$n$ with $\abs{V(G)}\leq\abs{V(K)}$. 
Considering the usual inclusion of~$G$ into the complete reflexive graph on $\abs{V(G)}$ vertices, we see that the $\Gamma$-size of~$G$ is precisely its number of vertices,
\[ s_\Gamma(G) = \abs{V(G)}. \]
We now compute the $\widehat{\Gamma}$-size of a persistent narrative in~$\Grefl$. 

\begin{lemma}\label{lemma:monomorphisms-in-Pe-T-Grefl}
	The monomorphisms in the category $\Pe(\timeT,\Grefl)$ are precisely the natural transformations whose components are injective maps of vertex sets.
\end{lemma}

\begin{proof}
	Given a natural transformation $\varphi\co X\Rightarrow Y$ with $X,Y\in\Ob(\Pe(\timeT,\Grefl))$, if $\varphi$ is a monomorphism in $\Pe(\timeT,\Grefl)$, then it is also a monomorphism in the functor category $[\timeT^\op,\Grefl]$ because 
	the inclusion $\Pe(\timeT,\Grefl)\hookrightarrow[\timeT,\Grefl]$ is a right adjoint. 
	Therefore, \cite[Corollary 2.15.3]{Borceux_Handbook-I} implies that each component $\varphi_t\co X(t)\to Y(t)$ is a monomorphism in~$\Grefl$. By the previous lemma, this means that $\varphi_t$ is injective as a map $\varphi_t\co V(X(t))\to V(Y(t))$.
	Conversely, we suppose that $\varphi_t$ is injective for every time interval $t\in\Ob(\timeT)$. Then by the same lemma, $\varphi_t\co X(t)\to Y(t)$ is monic in $\Grefl$ and thus in $[\timeT,\Grefl]$. Since $\Pe(\timeT,\Grefl)$ was defined as a full subcategory of $[\timeT,\Grefl]$, $\varphi$~is a morphism in $\Pe(\timeT,\Grefl)$, and it is obvious that it is also a monomorphism there.
\end{proof}

This shows that the $\widehat{\Gamma}$-size $s_{\widehat{\Gamma}}(X)$ of a persistent narrative~$X$ in $\Grefl$ on time~$\timeT$ is the minimum integer $n\geq0$ for which there exists an object $F\in\Ob(\widehat{\Omega}_n)$ together with a natural transformation $\varphi\co X\Rightarrow F$ whose component maps are injective. Recall that the objects of~$\widehat{\Omega}_n$ are the persistent narratives $F\co \timeT^\op\to\Grefl$ such that for every $s\in\Ob(\timeS)$, the finite reflexive graph~$F(s)$ lies in~$\Omega_n$, that is, $F(s)$~is complete and has order at most~$n$.
Thus, $s_{\widehat{\Gamma}}(X)$ is the minimum integer $n\geq0$ for which there is a persistent narrative $F\co \timeT^\op\to\Grefl$ along with a natural transformation $\varphi\co X\Rightarrow F$ such that for every $t\in\Ob(\timeT)$, the map $\varphi_t\co V(X(t))\to V(F(t))$ is injective, and for every $s\in\Ob(\timeS)$, $F(s)$~is a complete reflexive graph of order at most~$n$.

\begin{prop}\label{prop:case-study-ordinary-tree-width}
	The $\widehat{\Gamma}$-size of a persistent narrative $X\co \timeT^\op\to\Grefl$ is the maximum order of the reflexive graphs~$X(s)$ with $s\in\Ob(\timeS)$,
	\[ s_{\widehat{\Gamma}}(X) = \max_{s\in\Ob(\timeS)} \abs{V(X(s))}. \]
\end{prop}

\begin{proof}
	To show the inequality~$\geq$, let $n$ be any non-negative integer, and let $F\co \timeT^\op\to\Grefl$ be any persistent narrative and $\varphi\co K\Rightarrow F$ any natural transformation with the properties that $\varphi_t\co V(X(t))\to V(F(t))$ is injective for all $t\in\Ob(\timeT)$ and the finite reflexive graph~$F(s)$ is complete of order at most~$n$ for all $s\in\Ob(\timeS)$. Then for each $s\in\Ob(\timeS)$, we have $n\geq\abs{V(F(s))}\geq\abs{V(X(s))}$, by injectivity of $\varphi_s$, which implies that $s_{\widehat{\Gamma}}(X)\geq \max_{s\in\Ob(\timeS)} \abs{V(X(s))}$.
	
	To prove the reverse inequality, we define $F(t)$ for every $t\in\Ob(\timeT)$ to be the complete reflexive graph on the set $V(F(t))\coloneqq V(X(t))$. For every inclusion map $i\co t\hookrightarrow t'$ in~$\timeT$, we put $F(i)\coloneqq X(i)\co V(F(t'))=V(X(t'))\to V(X(t))=V(F(t))$. Note that completeness and reflexivity of $F(t)$ ensures that $F(i)$ is a morphism of finite reflexive graphs $F(i)\co F(t')\to F(t)$. (Note that reflexivity is essential for the case that $X(t')$ has two distinct non-adjacent vertices $v$ and~$w$ which have the same image $x\in V(X(t))$ under $X(i)$. Since $v$ and $w$ are adjacent in~$F(t')$, the existence of a loop at~$x$ in~$F(t)$ is a necessary condition for $F(i)$ to be a graph morphism since graph morphisms are not allowed to smash edges into single vertices. This is also the reason why we have replaced the category of ordinary simple graphs considered in \cite[Section~3.1]{Bumpus-et-al_Structured-Decompositions} by reflexive graphs.) 
	
	By functoriality of $X\co \timeT^\op\to\Grefl$, the above defines a functor $F\co\timeT^\op\to\Grefl$. It is a persistent narrative in $\Grefl$:
	Given any square in~$\timeT$ of shape
	\begin{equation*}
		\begin{tikzcd}
			{[p,p]} \arrow[r, "j", hook] \arrow[d, "i", hook, swap] & {[p,b]} \arrow[d, "\ell", hook]
			\\
			{[a,p]} \arrow[r, "k", hook, swap] & {[a,b]},
		\end{tikzcd}
	\end{equation*}
	applying the persistent narrative~$X$ provides a pullback square in $\Grefl$ of the form
	\begin{equation*}
		\begin{tikzcd}
			X([a,b]) \arrow[r, "X(j)"] \arrow[d, "X(i)", swap] \arrow[rd, phantom, "\mathlarger{\lrcorner}", very near start] 
			& X([p,b]) \arrow[d, "X(\ell)"]
			\\
			X([a,p]) \arrow[r, "X(k)", swap] 
			& X([p,p]),
		\end{tikzcd}
	\end{equation*}
	whose underlying square of vertex sets is a pullback square in the category of finite sets:
		\begin{equation*}
			\begin{tikzcd}
				V(X([a,b])) \arrow[r, "X(j)"] \arrow[d, "X(i)", swap] \arrow[rd, phantom, "\mathlarger{\lrcorner}", very near start] 
				& V(X([p,b])) \arrow[d, "X(\ell)"]
				\\
				V(X([a,p])) \arrow[r, "X(k)", swap] 
				& V(X([p,p])).
			\end{tikzcd}
		\end{equation*}
	By definition of $F\co\timeT^\op\to\Grefl$, this square becomes 
	\begin{equation*}
		\begin{tikzcd}
			V(F([a,b])) \arrow[r, "F(j)"] \arrow[d, "F(i)", swap] \arrow[rd, phantom, "\mathlarger{\lrcorner}", very near start] 
			& V(F([p,b])) \arrow[d, "F(\ell)"]
			\\
			V(F([a,p])) \arrow[r, "F(k)", swap] 
			& V(F([p,p])).
		\end{tikzcd}
	\end{equation*}
	This in turn induces a pullback square in $\Grefl$:
	\begin{equation*}
		\begin{tikzcd}
			F([a,b]) \arrow[r, "F(j)"] \arrow[d, "F(i)", swap] \arrow[rd, phantom, "\mathlarger{\lrcorner}", very near start] 
			& F([p,b]) \arrow[d, "F(\ell)"]
			\\
			F([a,p]) \arrow[r, "F(k)", swap] 
			& F([p,p]),
		\end{tikzcd}
	\end{equation*}
	which shows that $F\in\Ob(\Pe(\timeT,\Grefl))$.
	By construction, the reflexive graph $F(s)$ is complete for each $s\in\Ob(\timeS)$. Moreover, we have $\abs{V(F(s))} = \abs{V(X(s))} \leq \max_{s'\in\Ob(\timeS)} \abs{V(X(s'))}$, so $F(s)$ is a complete reflexive graph of order at most $\max_{s\in\Ob(\timeS)} \abs{V(X(s))}$.
	
	It remains to show that there exists a natural transformation $\varphi\co X\Rightarrow F\co\timeT^\op\to\Grefl$ such that for every $t\in\Ob(\timeT)$, the map $\varphi_t\co V(X(t))\to V(F(t))$ is injective.
	By construction, we have an identity set map $\varphi_t\coloneqq\id{}\co V(X(t))\to V(F(t))$, and since $F(t)$ is reflexive and complete, this is a morphism $\varphi_t\co X(t)\to F(t)$ in $\Grefl$ (but not in general the identity morphism since $X(t)$ is not necessarily complete). The family $\varphi\coloneqq (\varphi_t)_{t\in\Ob(\timeT)}$ is a natural transformation $\varphi\co X\Rightarrow F$ because given any inclusion map $i\co t\hookrightarrow t'$ in~$\timeT$, then by construction of $F(i)=X(i)$ and $\varphi_t=\id{V(X(t))}$, $\varphi_{t'}=\id{V(X(t'))}$, we have an obvious commutative square of finite sets
	\begin{equation*}
	\begin{tikzcd}
		V(X(t')) \arrow[r, "X(i)"] \arrow[d, "\varphi_{t'}", swap] & V(X(t)) \arrow[d, "\varphi_t"] \\
		V(F(t')) \arrow[r, "F(i)", swap] & V(F(t)).
	\end{tikzcd}
	\end{equation*}
	Thus, the corresponding square in $\Grefl$
	\begin{equation*}
	\begin{tikzcd}
		X(t') \arrow[r, "X(i)"] \arrow[d, "\varphi_{t'}", swap] & X(t) \arrow[d, "\varphi_t"] \\
		F(t') \arrow[r, "F(i)", swap] & F(t)
	\end{tikzcd}
	\end{equation*}
	commutes, as well.
	We now have a natural transformation $\varphi\co X\Rightarrow F\co \timeT^\op\to\Grefl$ from $X$ to the persistent narrative $F\co\timeT^\op\to\Grefl$. For every $t\in\Ob(\timeT)$, the map $\varphi_t\co V(X(t))\to V(F(t))$ is the identity map and therefore injective, and for every $s\in\Ob(\timeS)$, $F(s)$~is a complete reflexive graph with $\abs{V(F(s))} = \abs{V(X(s))} \leq \max_{s'\in\Ob(\timeS)} \abs{V(X(s'))}$.
	This proves the desired inequality $s_{\widehat{\Gamma}}(X) \leq \max_{s'\in\Ob(\timeS)} \abs{V(X(s'))}$.
\end{proof}

Hence, we have shown that the $\widehat{\Gamma}$-size of a persistent narrative $X\co\timeT^\op\to\Grefl$ is given by
\[ s_{\widehat{\Gamma}}(X) = \max_{s\in\Ob(\timeS)} \abs{V(X(s))} = \max_{s\in\Ob(\timeS)} s_\Gamma(X(s)), \]
that is, it is the maximum $\Gamma$-size of the graphs~$X(s)$ with $s\in\Ob(\timeS)$.

We now pass on to the notions of $\Gamma$-width and $\widehat{\Gamma}$-width.
Let $J$ be a tree. The $\Gamma$-width of a structured decomposition $d\co \intJ\to\Grefl$ is
\[ w_\Gamma(d) = \max_{v\in V(J)} s_\Gamma (d(v)) - 1 = \max_{v\in V(J)} \abs{d(v)}, \]
so this is the classical notion of width of a tree-decomposition from ordinary graph theory.
Given a structured decomposition of the form $d\co\intJ\to\Pe(\timeT,\Grefl)$, its $\widehat{\Gamma}$-width is given by
\begin{align*}
	w_{\widehat{\Gamma}}(d) 
	&= \max_{v\in V(J)} s_{\widehat{\Gamma}}(d(v)) - 1 \\
	&= \max_{v\in V(J)} \max_{s\in\Ob(\timeS)} s_\Gamma(d(v)(s)) - 1 \\
	&= \max_{s\in\Ob(\timeS)} \sBigl ( \max_{v\in V(J)} \abs{d_s(v)}-1 \sBigr ) \\
	&= \max_{s\in\Ob(\timeS)} w_\Gamma(d_s)
\end{align*}
where for each $s\in\Ob(\timeS)$, $d_s\co\intJ\to\Grefl$ denotes the composite functor $\intJ\xrightarrow{d}\Pe(\timeT,\Grefl)\hookrightarrow[\timeT^\op,\Grefl]\xrightarrow{\textup{ev}_s}\Grefl$.
Thus, the $\widehat{\Gamma}$-width $w_{\widehat{\Gamma}}(d)$ of a tree-shaped structured decomposition $d\co\intJ\to\Grefl$ in $\Grefl$ is the maximum $\Gamma$-width $\max_{s\in\Ob(\timeS)} w_\Gamma(d_s)$ of the component structured decompositions $d_s\co \intJ\to\Grefl$ over all $s\in\Ob(\timeS)$.
As a consequence, we recover the maximum tree-width of temporal graphs, 
where the maximum refers to the sub-join-semilattice $\timeS\subseteq\timeT$. Varying~$\timeS$ thus provides a flexible tool to modify the resulting measure of time-dependent temporal tree-resemblance of temporal graphs.
For example, if $\timeS$ is the sub-join-semilattice containing all intervals of length at least~$k$ for some fixed 
integer $k\geq0$, then the temporalized tree-width~$w_{\widehat{\Gamma}}(d)$ only respects snapshots of sufficient length.

\subsection{Case study: Complemented tree-width}

We now pass on to the complemented tree-width. In \cite[Section~3.3]{Bumpus-et-al_Structured-Decompositions}, the authors construct a spined sd-category to capture the measure of complemented tree-width of a finite simple undirected graph~$G$, which is defined as the ordinary tree-width of the \textit{complement}~$\overline{G}$ of~$G$, that is, the graph with the same vertex set $V(\overline{G})\coloneqq V(G)$, and for any two distinct vertices $v$ and~$w$, an edge $vw\in E(\overline{G})$ iff $vw$ is not an edge of~$G$.
Using our theory, we will now establish a time-respecting version of complemented tree-width, by temporalizing the spined sd-category of \cite[Proposition~3.3.4]{Bumpus-et-al_Structured-Decompositions} and computing the respective width.
First, we define a category~$\overline{\Grph}$ as follows. (See also \cite[Definition~3.3.1]{Bumpus-et-al_Structured-Decompositions}.)
The objects of~$\overline{\Grph}$ are the finite simple undirected graphs (without loops and without multiple edges). As usual, we will simply refer to them as \textit{graphs}. The morphisms of $\overline{\Grph}$ from $G$ to~$H$ are the graph morphisms from $\overline{G}$ to~$\overline{H}$, that is, the set maps $f\co V(G)\to V(H)$ such that for any $v,w\in V(G)$, if $f(v)$ and $f(w)$ are joint by an edge in~$H$, then $v$ and $w$ are joint by an edge in~$G$. We refer to the morphisms of~$\overline{\Grph}$ as \textit{complement (graph) morphisms}. 
The composition and identity morphisms in $\overline{\Grph}$ are defined in the obvious way.
For every integer $n\geq0$, let $\Omega_n$ denote the full subcategory of $\overline{\Grph}$ whose objects are the edgeless graphs of order at most~$n$. As usual, we put $\Omega\coloneqq(\Omega_n)_{n\geq0}$ and write $\iota_n\Omega_n\hookrightarrow\overline{\Grph}$ for the inclusion functors.
By \cite[Proposition 3.3.4]{Bumpus-et-al_Structured-Decompositions}, the triple $\Gamma\coloneqq(\overline{\Grph}, \{\textup{trees}\},\Omega)$ is a spined sd-category.
Beware that it does not satisfy all assumptions of Theorem~\ref{thm:temporalization-of-spined-sd-category}. However, $\overline{\Grph}$ can be embedded into a presheaf category, which meets the requirements, so we may apply the theorem.
Now, for every sub-join-semilattice $\timeS$ of a finite discrete time category~$\timeT$, we obtain a spined sd-category $\widehat{\Gamma}\coloneqq \sBigl( \Pe(\timeT,\overline{\Grph}),\{\textup{trees}\},\widehat{\Omega} \sBigr)$.
We compute the $\Gamma$-width of a graph. (See also \cite[Proposition 3.3.4]{Bumpus-et-al_Structured-Decompositions}.)

\begin{lemma}\label{lemma:monomorphisms-in-overline-Grph}
	A complement morphism $f\co G\to H$ of graphs is a monomorphism in $\overline{\Grph}$ if and only if the map $f\co V(G)\to V(H)$ is injective.
\end{lemma}

\begin{proof}
	If $f\co V(G)\to V(H)$ is injective, then $f\co G\to H$ is obviously a monomorphism in $\overline{\Grph}$ so suppose that $f$ is monic, and let $v,w\in V(G)$ with $f(v)=f(w)$. Let $J$ be the graph with precisely one vertex~$x$, and define graph morphisms $g,h\to J\to G$ by $g(x)\coloneqq v$ and $h(x)\coloneqq w$, respectively. We have $fg=fh$, so $g=h$, and thus, $v=w$.
\end{proof}

Therefore, the $\Gamma$-size $s_\Gamma(G)$ of a graph~$G$ is the minimum integer $n\geq0$ such that there is a graph~$Q$ in $\Omega_n$ together with a complement graph morphism $f\co G\to Q$ which is injective as a set map $f\co V(G)\to V(Q)$. Since $Q$ has no edges, any such map is a complement graph morphism, and by finiteness of the sets $V(G)$ and $V(Q)$, there exists an injective map from $V(G)$ to $V(Q)$ if and only if $\abs{V(G)}\leq\abs{V(Q)}$. Thus, $s_\Gamma(G)$ is the minimum $n\geq0$ such that there is an edgeless graph~$Q$ of order at most~$n$ satisfying $\abs{V(G)}\leq\abs{V(Q)}$. Considering the usual inclusion of~$G$ into the edgeless graph with vertex set~$V(G)$, we see that 
\[ s_\Gamma(G) = \abs{V(G)}. \]
Let us now compute the $\widehat{\Gamma}$-size of a persistent narrative in $\overline{\Grph}$. 

\begin{lemma}
	The monomorphisms in the category $\Pe(\timeT,\overline{\Grph})$ are exactly the natural transformations whose component maps are injective.
\end{lemma}

\begin{proof}
	The argument is the same as for Lemma~\ref{lemma:monomorphisms-in-Pe-T-Grefl}. 
\end{proof}

Hence, the $\widehat{\Gamma}$-size $s_{\widehat{\Gamma}}(X)$ of a persistent narrative $X\co \timeT^\op\to\overline{\Grph}$ is the minimum non-negative integer~$n$ for which there exist an object $F$ of~$\widehat{\Omega}_n$ and a natural transformation $\varphi\co X\Rightarrow F$ whose component maps are injective. In other words, $s_{\widehat{\Gamma}}(X)$ is the minimum $n\geq0$ for which there exists a persistent narrative $F\co \timeT^\op\to\overline{\Grph}$ together with a natural transformation $\varphi\co X\Rightarrow F$ such that for each $t\in\Ob(\timeT)$, the map $\varphi_t\co V(X(t))\to V(F(t))$ is injective, and for each $s\in\Ob(\timeS)$, the graph $F(s)$ is edgeless and of order at most~$n$.

\begin{prop}
	The $\widehat{\Gamma}$-size of a persistent narrative $X\co \timeT^\op\to\overline{\Grph}$ is the maximum order of the graphs~$X(s)$ with $s\in\Ob(\timeS)$,
	\[ s_{\widehat{\Gamma}}(X) = \max_{s\in\Ob(\timeS)} \abs{V(X(s))}. \]
\end{prop}

\begin{proof}
	The inequality~$\geq$ can be proved in the same manner as for Proposition~\ref{prop:case-study-ordinary-tree-width}.
	Regarding the reverse inequality, we define $F(t)$ for each $t\in\Ob(\timeT)$ to be the graph with vertex set $V(F(t))\coloneqq V(X(t))$ and empty edge set. For each inclusion map $i\co t\to t'$ in~$\timeT$, we put $F(i)\coloneqq X(i)\co V(F(t'))=V(X(t'))\to V(X(t))=V(F(t))$. Since $F(t)$ has no edges, $F(i)$~is a complement graph morphism. (Note that in contrast to the case of Proposition~\ref{prop:case-study-ordinary-tree-width}, we are not allowed to suppose that $F(t)$ is a reflexive graph instead of a mere graph: If $X(i)$ maps to non-adjacent vertices $v$ and~$w$ of~$X(t')$ to the same vertex~$x$ of $X(t)$, then a loop at~$x$ in~$F(t)$ would force an edge between $v$ and~$w$ in $F(t')$ to ensure that $F(i)$ is a complement morphism.)
	
	A similar argument as in the proof of Proposition~\ref{prop:case-study-ordinary-tree-width} shows that the above defines a persistent narrative $F\co\timeT^\op\to\overline{\Grph}$ 
	and that the identity set maps $\varphi_t\coloneqq\id{}\co V(X(t))\to V(F(t))$ assemble into a natural transformation $\varphi\co X\Rightarrow F$.
	To sum up, we have found a persistent narrative~$F$ on~$\timeT$ with values in $\overline{\Grph}$ and a natural transformation~$\varphi$ from $X$ to~$F$ with injective component maps, and for each $s\in\Ob(\timeS)$, the graph~$F(s)$ has no edges and has order $\abs{V(F(s))}=\abs{V(X(s))}\leq\max_{s'\in\Ob(\timeS)} \abs{V(X(s'))}$. We conclude that $s_{\widehat{\Gamma}}(X)\leq\max_{s\in\Ob(\timeS)} \abs{V(X(s))}$.
\end{proof}

Thus, the $\widehat{\Gamma}$-size of a persistent narrative $X\co \timeT\to\overline{\Grph}$ is the maximum $\Gamma$-size of the graphs~$X(s)$ with $s\in\Ob(\timeS)$:
\[ s_{\widehat{\Gamma}}(X) = \max_{s\in\Ob(\timeS)} \abs{V(X(s))} = \max_{s\in\Ob(\timeS)} s_\Gamma(X(s)). \]

To compute the $\Gamma$-width~$w_\Gamma(d)$ of a structured decomposition $d\co \intJ\to\overline{\Grph}$ for any tree~$J$, note that this induces a $J$-structured decomposition $\overline{d}\co \intJ\to\Grph$ in~$\Grph$ given by $\overline{d}(x)\coloneqq\overline{d(x)}$ for all objects and morphisms~$x$ in $\intJ$. Thus, denoting by~$\Gamma'$ the spined sd-category $(\Grph,\{\textup{trees}\},(\{\textup{complete graphs of order }\leq n\}))$ of \cite[Proposition 3.3.1]{Bumpus-et-al_Structured-Decompositions}, we have
\[ w_\Gamma(d) = \max_{v\in V(J)} s_\Gamma(d(v)) = \max_{v\in V(J)} s_{\Gamma'}(\overline{d(v)}) = \textup{tw}(\overline{d}), \]
which is the ordinary notion of complemented width of a tree-decomposition from graph theory.
Analogously as in the previous case study, we deduce that the $\widehat{\Gamma}$-width of a structured decomposition $d\co\intJ\to\Pe(\timeT,\overline{\Grph})$ is given by
\[ w_{\widehat{\Gamma}}(d) = \max_{s'\in\Ob(\timeS)} w_\Gamma (d_s), \]
that is, the maximum $\Gamma$-width of the component structured decompositions $d_s\co\intJ\to\overline{\Grph}$ over all $s\in\Ob(\timeS)$.

\subsection{Case study: Tree independence number}

In our last case study, we deal with the tree independence number, based on \cite[Section~3.4]{Bumpus-et-al_Structured-Decompositions}. We take again the category $\overline{\Grph}$ of graphs and complement morphisms. For each integer $n\geq0$, we define $\Omega_n^\alpha$ to be the full subcategory of $\overline{\Grph}$ whose objects are the graphs of independence number at most~$n$, where the \textit{independence number}~$\alpha(G)$ of a graph~$G$ is defined as the maximum cardinality of a set of vertices of~$G$ that are pairwise non-adjacent. By \cite[Proposition 3.4.5]{Bumpus-et-al_Structured-Decompositions}, the triple $\Gamma\coloneqq(\overline{\Grph}, \{\textup{trees}\},\Omega^\alpha)$ with $\Omega^\alpha\coloneqq (\Omega_n^\alpha)_{n\in\N}$ is a spined sd-category, and the $\Gamma$-width $w_\Gamma(G)$ of a graph~$G$ is precisely its \textit{tree independence number}~$\textup{tw}_\alpha(G)$. This is the minimum \textit{tree independence number} $\textup{tw}_\alpha(d)$ of the tree decompositions~$d$ of~$G$, where $\textup{tw}(d)$ is defined as
\[ \textup{tw}_\alpha(d) \coloneqq \max_{t\in V(T)} \alpha(d(t)) -1. \]

Given a sub-join-semilattice~$\timeS$ of a finite discrete time category~$\timeT$, we again obtain a spined sd-category $\widehat{\Gamma}\coloneqq \sBigl( \Pe(\timeT,\overline{\Grph}),\{\textup{trees}\},\widehat{\Omega}^\alpha \sBigr)$.
A similar reasoning as in the previous case study shows that the $\Gamma$-size $s_\Gamma(G)$ of a graph~$G$ is precisely its independence number~$\alpha(G)$ and that the $\widehat{\Gamma}$-size $s_{\widehat{\Gamma}}(X)$ of a persistent narrative~$X$ on~$\timeT$ with values in $\overline{\Grph}$ is the maximum independence number of the graphs $X(s)$ with $s\in\Ob(\timeS)$: 

\begin{prop}
	For any persistent narrative $X\co \timeT^\op\to\overline{\Grph}$, its $\widehat{\Gamma}$-size is
	\[s_{\widehat{\Gamma}}(X) = \max_{s\in\Ob(\timeS)} \alpha(X(s)) = \max_{s\in\Ob(\timeS)} s_\Gamma(X(s)). \]
\end{prop}

Let $J$ be a tree, and let $d\co\intJ\to\Pe(\timeT,\overline{\Grph})$ be a structured decomposition. By an analogous argument as in the case study~\ref{subsection:case-study-ordinary-tree-width}, the $\widehat{\Gamma}$-width $w_{\widehat{\Gamma}}(d)$ of~$d$ is the maximum $\Gamma$-width of the component structured decompositions $d_s\co\intJ\to\overline{\Grph}$ over all $s\in\Ob(\timeS)$, that is, the maximum tree independence number of $d_s$ for $s\in\Ob(\timeS)$,
\[ w_{\widehat{\Gamma}}(d) = \max_{s\in\Ob(\timeS)} w_\Gamma(d_s) = \max_{s\in\Ob(\timeS)} \textup{tw}_\alpha(d_s). \]

\section{Conclusion and future directions}
\label{section:conclusion}
	
Time-varying data is ubiquitous in many different applications. However, one is often confronted with at least one of the following two challenges: (1) analyzing and comparing data collected in several locations; (2) performing calculations on the temporal data. Both of these situations require a notion of time-varying data that is decomposed into smaller components, as suggested by our temporalized version of spined sd-categories. Having established the formal framework, we may now turn to various interesting prospective questions. Some specific future directions include the following.

\paragraph{The cumulative perspective.}
\label{paragraph:cumulative}

In their paper~\cite{Bumpus-et-al_Time-Varying}, the authors do not only introduce persistent narratives, but also cumulative ones.
Given a discrete time category~$\timeT$ and a category~$\cC$ which admits all pushouts, a \textit{cumulative narrative on~$\timeT$ with values in~$\cC$} is a co-presheaf $F\co \timeT\to\cC$ taking any square in~$\timeT$ as shown on the left-hand side below to a pushout square as shown on the right-hand side below:
\begin{equation*}
	\begin{tikzcd}
		{[p,p]} \arrow[r, hook] \arrow[d, hook] & {[p,b]} \arrow[d, hook]
		\\
		{[a,p]} \arrow[r, hook] & {[a,b]}
	\end{tikzcd}
	\qquad \qquad \qquad
	\begin{tikzcd}
		F([p,p]) \arrow[r] \arrow[d] \arrow[rd, phantom, "\mathlarger{\ulcorner}", very near end] 
		& F([p,b]) \arrow[d]
		\\
		F([a,p]) \arrow[r] 
		& F([a,b]).
	\end{tikzcd}
\end{equation*}
We denote the full subcategory of the category $[T,\cC]$ of co-presheaves $\timeT\to\cC$ whose objects are the $\cC$-valued cumulative narratives on~$\timeT$ by $\Cu(\timeT,\cC)$.
It would be desirable to temporalize structured decompositions and width also from the cumulative point of view. 
More specifically, under dual assumptions of Theorem~\ref{thm:termporalized-sd-category}, we may form a pullback square
\begin{equation*}
	\begin{tikzcd}
		\Cu(\timeT,\cC)\times_{\Cu(\timeS,\cC)}\Cu(\timeS,\Omega_n) \arrow[r] \arrow[d, "\pi_n", swap]
		\arrow[dr, phantom, "\usebox\pullbacksquare", very near start, color=black]
		& \Cu(\timeS,\Omega_n) \arrow[d, "{\Cu(\timeS,\iota_n)}"]
		\\ \Cu(\timeT,\cC) \arrow[r, "{\Cu(\tau,\cC)}", swap]
		& \Cu(\timeS,\cC).
	\end{tikzcd}
\end{equation*}
We define $\check{\Omega}_n\subseteq \Cu(\timeT,\cC)$ to be the image category of the functor~$\pi_n$. 
If we dualize the assumptions of Theorem~\ref{thm:temporalization-of-spined-sd-category}, do we again obtain a spined sd-category of the form $(\Cu(\timeT,\cC),\cG,\check{\Omega})$? 
How does the proof change?

\paragraph{Relation between persistent and cumulative temporalization.}
\label{paragraph:relation-persistent-cumulative}

Let us assume that the above-mentioned procedure indeed provides a well-defined spined sd-category $(\Cu(\timeT,\cC),\cG,\check{\Omega})$. We may then wonder: How do the persistent and the cumulative temporalizations of spined sd-categories interact with each other? 

Given any time category~$\timeT$ and a both complete and cocomplete category~$\cC$, then by \cite[Theorem~2.10]{Bumpus-et-al_Time-Varying}, there exists a pair of adjoint functors 
\begin{equation*}
	\begin{tikzcd}
		\Pe(\timeT,\cC)
		\arrow[bend left=30]{r}
		\arrow[r, phantom, "\raisebox{-0.2ex}{\rotatebox{270}{$\dashv$}}"]
		& \Cu(\timeT,\cC).
		\arrow[bend left=30]{l}
	\end{tikzcd}
\end{equation*} 
It would be interesting to investigate on the following questions, assuming the setting of Theorem~\ref{thm:temporalization-of-spined-sd-category} together with the dual assumptions:

Are the above adjoint functors sd-functors in the sense of \cite[Definition 2.7.1]{Bumpus-et-al_Structured-Decompositions}? 

Are they even width-preserving in the sense of \cite[Definition 2.7.8]{Bumpus-et-al_Structured-Decompositions}? 

If not, can this be guaranteed by modifying the assumptions appropriately?

It would be highly desirable to have an adjoint pair of width-preserving sd-functors between the persistent and cumulative temporalized spined sd-categories. This would not only reveal a fundamental relation between the persistent and cumulative perspectives but also allow a convenient transfer between concepts and results for persistent narratives to cumulative ones and vice versa.

\paragraph{Compatibility of narratives and structured decompositions.}

Let $\timeT$ be a finite discrete time category.
In Section~\ref{subsection:narratives}, we have described the category $\Pe(\timeT,\cC)$ of persistent narratives in a spined sd-category $(\cC,\cG,\Omega)$ with suitable additional structure itself as a spined sd-category. Applying the procedure of \cite[Definition 2.4.1]{Bumpus-et-al_Structured-Decompositions}, we obtain a category $\mathfrak{D}(\Pe(\timeT,\cC))$ of structured decompositions in $\Pe(\timeT,\cC)$, defined as a full subcategory of the category of diagrams.

In contrast, we can start with the category~$\mathfrak{D}(\cC)$ of structured decompositions in~$\cC$ and then consider persistent narratives in it, giving rise to the category $\Pe(\timeT, \mathfrak{D}(\cC))$.
We make the following conjecture.

\begin{conj}\label{thm:embedding-structured-decompositions-persistent-narratives}
	Under appropriate assumptions, the category of structured decompositions of persistent narratives in~$\cC$ embeds into the category of persistent narratives of structured decompositions in~$\cC$, in the sense that there exists a fully faithful functor which is injective on objects
	\[ \mathfrak{D}(\Pe(\timeT,\cC)) \hookrightarrow \Pe(\timeT,\mathfrak{D}(\cC)). \]
\end{conj} 

Investigating on this more thoroughly, we hope to find a specific description of ``appropriate assumptions'' ensuring this conjecture to become true.

\addcontentsline{toc}{section}{References}
\printbibliography
	
\end{document}